\documentclass[11pt]{article}
\usepackage{srcltx}
\usepackage{eurosym}
\usepackage{mathtools}
\usepackage{amsmath}
\usepackage{amsfonts}
\usepackage{amssymb}
\usepackage{amsthm}
\usepackage{comment}
\usepackage{stmaryrd}

\usepackage{graphicx}
\usepackage{mathrsfs}
\usepackage{xcolor}
\usepackage{exscale}
\usepackage{latexsym}
\usepackage{authblk}

\usepackage{enumerate}
\usepackage[shortlabels]{enumitem}
\usepackage{bookmark}
\usepackage{wasysym}
\usepackage[ddmmyyyy]{datetime}
\usepackage[margin=1in]{geometry}
\makeatletter
\g@addto@macro\normalsize{%
	\setlength\abovedisplayskip{4pt}
	\setlength\belowdisplayskip{4pt}
	\setlength\abovedisplayshortskip{4pt}
	\setlength\belowdisplayshortskip{4pt}
}
\numberwithin{equation}{section}
\everymath{\displaystyle}
\usepackage[capitalize,nameinlink]{cleveref}
\crefname{section}{Section}{Sections}
\crefname{subsection}{Subsection}{Subsections}
\crefname{condition}{Condition}{Conditions}
\crefname{hypothesis}{Hypothesis}{Conditions}
\crefname{lemma}{Lemma}{Lemmas}
\crefname{definition}{Definition}{Definitions}

\crefformat{equation}{\textup{#2(#1)#3}}
\crefrangeformat{equation}{\textup{#3(#1)#4--#5(#2)#6}}
\crefmultiformat{equation}{\textup{#2(#1)#3}}{ and \textup{#2(#1)#3}}
{, \textup{#2(#1)#3}}{, and \textup{#2(#1)#3}}
\crefrangemultiformat{equation}{\textup{#3(#1)#4--#5(#2)#6}}%
{ and \textup{#3(#1)#4--#5(#2)#6}}{, \textup{#3(#1)#4--#5(#2)#6}}%
{, and \textup{#3(#1)#4--#5(#2)#6}}

\Crefformat{equation}{#2Equation~\textup{(#1)}#3}
\Crefrangeformat{equation}{Equations~\textup{#3(#1)#4--#5(#2)#6}}
\Crefmultiformat{equation}{Equations~\textup{#2(#1)#3}}{ and \textup{#2(#1)#3}}
{, \textup{#2(#1)#3}}{, and \textup{#2(#1)#3}}
\Crefrangemultiformat{equation}{Equations~\textup{#3(#1)#4--#5(#2)#6}}%
{ and \textup{#3(#1)#4--#5(#2)#6}}{, \textup{#3(#1)#4--#5(#2)#6}}%
{, and \textup{#3(#1)#4--#5(#2)#6}}

\crefdefaultlabelformat{#2\textup{#1}#3}
\newtheorem{theorem} {Theorem}[section]

\newtheorem{lemma}[theorem]{Lemma}

\newtheorem{counter example}[theorem]{Counter Example}
\newtheorem{remark}[theorem] {Remark}
\newtheorem{definition}[theorem] {Definition}
\def\CC{{\rm \kern.24em \vrule width.02em height1.4ex depth-.05ex \kern-.26emC}}

\def\TagOnRight

\def\AA{{it I} \hskip-3pt{\tt A}}

\def\QQ{\rlap {\raise 0.4ex \hbox{$\scriptscriptstyle |$}} {\hskip -0.1em Q}}

\makeatletter
\newcommand{\vo}{\vec{o}\@ifnextchar{^}{\,}{}}
\makeatother
\def\YYint#1#2#3{{\setbox0=\hbox{$#1{#2#3}{\iint}$}
		\vcenter{\hbox{$#2#3$}}\kern-.50\wd0}}
\def\bm \sigmant#1{\mathchoice
	{\XXint\displaystyle\textstyle{#1}}%
	{\XXint\textstyle\scriptstyle{#1}}%
	{\XXint\scriptstyle\scriptscriptstyle{#1}}%
	{\XXint\scriptscriptstyle\scriptscriptstyle{#1}}%
	\!\int}
\def\XXint#1#2#3{{\setbox0=\hbox{$#1{#2#3}{\int}$}
		\vcenter{\hbox{$#2#3$}}\kern-.50\wd0}}

\makeatletter
\def\namedlabel#1#2{\begingroup
	\def\@currentlabel{#2}%
	\label{#1}\endgroup
}
\makeatother
\makeatletter
\newcommand{\rmh}[1]{\mathpalette{\raisem@th{#1}}}
\newcommand{\raisem@th}[3]{\hspace*{-1pt}\raisebox{#1}{$#2#3$}}
\makeatother

\newcounter{desccount}

\newcommand{\descref}[2]{\hyperref[#1]{\textnormal{\textcolor{black}{}\textcolor{blue}{ #2}\textcolor{black}{}}}}
\newcommand{\dref}[2]{\hyperref[#1]{\textcolor{black}{(}\textcolor{blue}{\bf #2}\textcolor{black}{)}}}
\newcommand{\be} {\begin{eqnarray}}
	\newcommand{\ee} {\end{eqnarray}}
\newcommand{\Bea} {\begin{eqnarray*}}
	\newcommand{\Eea} {\end{eqnarray*}}
\newcounter{whitney}
\refstepcounter{whitney}

\newcounter{ineqcounter}
\refstepcounter{ineqcounter}
\makeatletter
\def\ps@pprintTitle{%
	\let\@oddhead\@empty
	\let\@evenhead\@empty
	\def\@oddfoot{}%
	\let\@evenfoot\@oddfoot}
\makeatother
\usepackage[doublespacing]{setspace}
\usepackage[titletoc,toc,page]{appendix}

\makeatletter
\newcommand{\refcheckize}[1]{%
	\expandafter\let\csname @@\string#1\endcsname#1%
	\expandafter\DeclareRobustCommand\csname relax\string#1\endcsname[1]{%
		\csname @@\string#1\endcsname{##1}\wrtusdrf{##1}}%
	\expandafter\let\expandafter#1\csname relax\string#1\endcsname
}
\makeatother
\refcheckize{\cref}
\refcheckize{\Cref}


\makeatletter
\newcommand{\mainsectionstyle}{%
	\renewcommand{\@secnumfont}{\bfseries}
	\renewcommand\section{\@startsection{section}{2}%
		\z@{.5\linespacing\@plus.7\linespacing}{-.5em}%
		{\normalfont\bfseries}}%
}
\makeatother
\usepackage{pgf,tikz}
\usetikzlibrary{arrows}
\usetikzlibrary{decorations.pathreplacing}
\usepackage{xpatch}
\xpatchcmd{\MaketitleBox}{\hrule}{}{}{}% remove first horizontal rule (above abstract)
\xpatchcmd{\MaketitleBox}{\hrule}{}{}{}% remoce second horizonral rule (below keywords)

\date{}
\usepackage{scalerel}
\makeatletter

\usepackage{subcaption}
\usepackage{hyperref}
\usepackage{caption}
\usepackage{subcaption}

\usepackage[utf8]{inputenc}
\allowdisplaybreaks
\usepackage{amsmath}
\usepackage{graphicx}
\usepackage{mathtools} 
\usepackage{hyperref}       
\usepackage{url}      
\usepackage{mathrsfs}
\usepackage{graphicx}
\usepackage{array}
\usepackage{color} 
\usepackage{tabularx}
\usepackage{amsmath,mathdots,amsthm}
\usepackage{amssymb}
\usepackage{amsfonts}
\usepackage{xcolor}
\usepackage{bm}
\usepackage{soul}
\usepackage{float}
\usepackage{cite}

\hypersetup{
	colorlinks=true,                          
	linkcolor=blue, % equation, section link color
	citecolor=blue, % bib color
	urlcolor=black  % url color if any
} 
\usepackage{orcidlink}

 \numberwithin{equation}{section}

\numberwithin{equation}{section}
\newtheorem*{assumption}{Assumption ($\mathcal{A}$)}

\newcommand{\rel}{^{\epsilon}}

\newcommand{\eps}{\varepsilon}

\newcommand{\half}{\frac{1}{2}}

\title{A structure-preserving implicit-explicit method for a hyperbolic approximation of fourth-order PDEs}
\author[1]{Rahul Barthwal\thanks{\href{mailto:rahul.barthwal@mathematik.uni-stuttgart.de}{rahul.barthwal@mathematik.uni-stuttgart.de}}}
\author[2]{Rahuldev Ghorai\thanks{\href{mailto:rahuldev.ghorai@inria.fr}{rahuldev.ghorai@inria.fr}, 
\href{mailto:rghorai@unistra.fr}{rghorai@unistra.fr}}}
\affil[1]{\footnotesize Institute of Applied Analysis and Numerical Simulation, University of Stuttgart\\

Pfaffenwaldring 57, 70569 Stuttgart, Germany}
\affil[2]{\footnotesize Université de Strasbourg, CNRS, Inria, IRMA, F-67000, Strasbourg, France}
\hypersetup{
pdftitle={},
pdfauthor={},
pdfkeywords={},
}
\begin{document}
\maketitle

\begin{abstract}
We introduce a novel structure-preserving numerical method for a first-order hyperbolic approximation system, which approximates the solutions of general fourth-order partial differential equations. By employing an implicit-explicit (IMEX) splitting between the stiff and non-stiff terms, we rigorously prove that the proposed scheme is energy-consistent and positivity-preserving under a CFL-type condition. Furthermore, we show that the method remains robustly stable in asymptotic regimes, with a time-step restriction that is entirely independent of the relaxation parameters. Finally, we present a series of numerical examples for thin film equations to validate the theoretical properties and efficacy of the scheme.
\end{abstract}
\medskip 
\section{Introduction}
In this article, we develop an asymptotic-preserving, energy-stable, and positivity-preserving implicit-explicit scheme for a lower-order relaxation system, which approximates solutions of the Cauchy problem for the fourth-order nonlinear PDEs of the form
\begin{align}
\label{eq: main}
u_t + \nabla \cdot \left( \mathcal{M}(u)\nabla\big(\gamma \Delta u - \Pi(u)\big) \right) = 0
\quad \text{in } \Omega_T := (0,T) \times  \Omega.
\end{align}
In \eqref{eq: main}, $\Omega=\mathbb{T}^d$ denotes the $d$-dimensional Torus. We supplement \eqref{eq: main} with initial data of the form
\begin{align}\label{initial_data_main}
    u(0, \cdot)=u_0> 0.
\end{align} %Accordingly, the solution is subject to periodic boundary conditions in each spatial direction.

Equations of the form \eqref{eq: main} arise naturally in the modelling of thin-film flows, interfacial dynamics, and related diffusion-driven processes. Here, $u$ typically represents a nonnegative height, $\gamma>0$ is a capillarity coefficient, $\mathcal M(u)$ is a nonlinear mobility, and $\Pi(u)$ denotes a lower-order pressure or potential term. In thin-film applications, the structure of $\mathcal M$ is dictated by the boundary condition at the liquid-solid interface. For instance, the no-slip condition leads to the degenerate mobility $\mathcal M(u)=u^3$, while Navier-slip models give mobilities of the form $\mathcal M(u)=u^3+\varepsilon u^n,\, \varepsilon>0,\, n\in(0,3).$ The degeneracy of $\mathcal M$ at $u=0$, together with possible singularities of $\Pi$, makes the design and analysis of robust numerical schemes particularly challenging.

A central analytical feature of \eqref{eq: main} is its gradient-flow structure. Formally, the smooth solutions of \eqref{eq: main} dissipates the free energy
\begin{align}\label{limit_energy}
E[u] = \int_{\Omega} \left( \frac{\gamma}{2}|\nabla u|^2 + W(u) \right)\, d\mathbf{x},
\end{align}
where $W'(u)=\Pi(u)$. This energy structure plays a crucial role in the existence theory and qualitative analysis of fourth-order parabolic equations; see, for example, \cite{bernis1990higher,giacomelli1999fourth,giacomelli2014well,gnann2018navier,dai2016weak}. From the numerical point of view, however, the fourth-order nature of the equation, the nonlinear mobility, and the positivity constraint $u> 0$ create substantial difficulties in designing robust discretizations. In particular, standard discretizations may fail to preserve non-negativity and may also destroy the energy-dissipation mechanism that governs the continuous problem.

For this reason, considerable effort has been devoted to the construction of numerical methods that preserve the structural properties of \eqref{eq: main}. Positivity-preserving and energy/entropy-stable schemes for thin-film type equations have been proposed in several works, particularly for degenerate mobilities and singular pressure terms; see, for example, \cite{zhornitskaya1999positivity,grun2000nonnegativity, kim2024positivity}. These schemes are designed to ensure mass conservation, non-negativity of the discrete solution, and monotone decay of the discrete free energy and entropy. Such properties are not merely technical conveniences but are essential for obtaining physically meaningful approximations, especially in underresolved regimes. As pointed out in the review paper by Bertozzi \cite{bertozzi1998mathematics}, even when the analytical solution remains strictly positive, a generic numerical scheme may produce negative values if the mesh is not sufficiently fine. This observation motivates the development of structure-preserving methods that remain stable and reliable beyond the asymptotically resolved regime.

One approach to develop structure-preserving methods to approximate the solutions of the Cauchy problem \eqref{eq: main}-\eqref{initial_data_main} is to approximate the fourth-order equation \eqref{eq: main} with lower-order approximate systems. Such reformulations are attractive because they allow one to design numerical methods using tools developed for first-order or mixed-order systems, while still retaining the entropy/ energy structure of the original equation. Lower-order approximations and relaxation-type formulations have therefore become a useful strategy for studying higher-order nonlinear PDEs. More recently, hyperbolic relaxation systems have been proposed as effective approximations of fourth-order diffusion equations, providing a framework in which energy stability, asymptotic consistency, and efficient time discretization can be addressed simultaneously; see, for instance, \cite{barthwal2025hyperbolic, barthwal2026energy, dhaouadi2025first,giesselmann2025convergence} and references cited therein. In particular, structure-preserving asymptotically stable methods have been developed for these relaxation systems in recent years; see e.g. \cite{albi2020implicit, boscarino2017unified, ma2025uniform,jin1995relaxation}. 

In this direction, Barthwal et al. \cite{barthwal2026energy} introduced an energy-consistent hyperbolic relaxation system approximating the solutions of \eqref{eq: main}-\eqref{initial_data_main}. Their proposed system depends on positive parameters $\alpha, \beta(\alpha)$, and $\tau$, and is given by 
\begin{equation}\label{hyperbolic_system_with_alpha}
\begin{aligned}
     u^{\alpha, \tau}_t+\nabla\cdot \mathbf{q}^{\alpha, \tau}&=0,\\
     \dfrac{1}{\alpha}\psi^{\alpha, \tau}_t+\nabla\cdot \mathbf{q}^{\alpha, \tau}&=-w^{\alpha, \tau},\\
     \tau\mathbf{q}^{\alpha, \tau}_t+\nabla(\Pi(u^{\alpha, \tau})+\psi^{\alpha, \tau})&=-\dfrac{\mathbf{q}^{\alpha, \tau}}{\mathcal{M}(u^{\alpha, \tau})},\\
    \beta w^{\alpha, \tau}_t-\gamma \nabla\cdot \mathbf{p}^{\alpha, \tau}&=\psi^{\alpha, \tau},\\
    \mathbf{p}^{\alpha, \tau}_t-\nabla w^{\alpha, \tau}&=0,
\end{aligned} \quad \mathrm{in}\quad \Omega_T.
\end{equation}
In this article, we particularly focus on one parameter family of \eqref{hyperbolic_system_with_alpha} as our main system. For simplicity, we set $1/\alpha=\tau=\beta=\epsilon$, where $\epsilon\in(0,1]$ acts as a single relaxation parameter. The system \eqref{hyperbolic_system_with_alpha} then reduces to
\begin{equation}\label{hyperbolic_system}
\begin{aligned}
u^{\epsilon}_t+\nabla\cdot \mathbf{q}^{\epsilon}&=0,\\
\epsilon \psi^{\epsilon}_t+\nabla\cdot \mathbf{q}^{\epsilon}&=-w^{\epsilon},\\
\epsilon\mathbf{q}^{\epsilon}_t+\nabla(\Pi(u^{\epsilon})+\psi^{\epsilon})&=-\dfrac{\mathbf{q}^{\epsilon}}{\mathcal{M}(u^{\epsilon})},\\
\epsilon w^{\epsilon}_t-\gamma \nabla\cdot \mathbf{p}^{\epsilon}&=\psi^{\epsilon},\\
\mathbf{p}^{\epsilon}_t-\nabla w^{\epsilon}&=0,
\end{aligned} \quad \mathrm{in}\quad \Omega_T.
\end{equation}
In the formal stiff relaxation limit ($\epsilon\to 0$), the auxiliary variables naturally recover the components of the fourth-order limit equation. Precisely, $\psi^\epsilon$ approximates $-\gamma\Delta u$, $\mathbf q^\epsilon$ approximates the nonlinear flux $-\mathcal M(u)\nabla\bigl(\Pi(u)-\gamma\Delta u\bigr),$  $w^\epsilon$ approximates $u_t$, and $\mathbf p^\epsilon$ approximates $\nabla u$. 

Crucially, the hyperbolic system \eqref{hyperbolic_system} inherits an energy structure analogous to the limit equation \eqref{eq: main}, defined by 
\begin{equation}\label{convex-energy}
    E^{\epsilon} =\displaystyle\int_{\Omega} \left(W(u\rel)+\dfrac{\epsilon({\psi\rel})^2}{2}+\dfrac{\epsilon {|\mathbf{q}\rel|}^2}{2}+\dfrac{\epsilon (w\rel)^2}{2}+\dfrac{\gamma {|\mathbf{p}\rel|}^2}{2}\right)\, d\mathbf{x}. 
\end{equation}
As $\epsilon \to 0$, $E^{\epsilon}$ formally converges to the limit energy defined in \eqref{limit_energy}, demonstrating that the system \eqref{hyperbolic_system_with_alpha} or equivalently \eqref{hyperbolic_system} is an energy-consistent relaxation system. 

Barthwal et al. \cite{barthwal2026energy} rigorously justified the convergence of weak entropy solutions of the Cauchy problem \eqref{hyperbolic_system_with_alpha}-\eqref{initial_data_main} using the relative energy framework. Under certain structural conditions on the mobility function $\mathcal{M}(u)$ and the pressure function $\Pi(u)$ (see Assumptions (\ref{A.1})-(\ref{A.2}) below), they proved that the Bregman distance between the weak entropy solutions of the relaxation system and the sufficiently smooth solutions of \eqref{eq: main} vanishes as the relaxation parameters tend to their stiff relaxation limit. This establishes that the solutions of the proposed system converge to those of the limit equation for the maximal time of existence of smooth solutions. They further validated their analytical results using tailored numerical solvers for hyperbolic PDEs.

However, the numerical scheme used in \cite{barthwal2026energy} does not preserve the structural properties such as energy inequality, positivity, and conservation properties of the continuous relaxation system at the fully discrete level. This significantly limits the applicability of the tailored solvers to more challenging regimes and broader physical settings. It is therefore crucial to develop numerical methods for \eqref{hyperbolic_system} that are provably structure-preserving. In particular, for singularly perturbed systems such as \eqref{hyperbolic_system}, a numerical scheme should satisfy the Asymptotic Preserving (AP) property \cite{Jin99}, which ensures that the scheme remains stable and consistent in the stiff relaxation limit $\epsilon\to0$ and, ideally, provides a consistent discretization of the limiting system without requiring the computational mesh to resolve the small relaxation scale.

A particularly effective strategy for constructing AP schemes is based on Implicit--Explicit (IMEX) time discretizations, in which the stiff terms are treated implicitly while the non-stiff terms are handled explicitly; see, for example, \cite{dimarco2018second, BAL+14, BLY17, BRS18, CDK12, DT11, DP13, HJL12, Kle95, NBA+14}. Such a treatment relaxes the severe CFL restrictions associated with fully explicit methods in the stiff regime, while avoiding the computational cost of solving the large, typically dense nonlinear systems arising from fully implicit discretizations.

It is important to point out that, for the thin-film applications governed by \eqref{eq: main} and approximated by \eqref{hyperbolic_system}, asymptotic preservation alone is not sufficient. The numerical method must also retain the fundamental physical and mathematical structures of the underlying system. In particular, positivity preservation is essential, since the primary unknown $u$ represents the local thickness of the fluid film. A loss of positivity may lead to nonphysical states, singular coefficients, and ultimately numerical breakdown. Although AP schemes for a wide range of singularly perturbed systems are well established, results that simultaneously incorporate such physical constraints remain comparatively scarce. Notable examples of structure-preserving AP methods include \cite{AGK23} for the stiff barotropic Euler system and \cite{HSZ18} for the stiff BGK equation.

These considerations motivate the development of a numerical scheme for singularly perturbed systems of the form \eqref{hyperbolic_system} that is simultaneously asymptotic preserving, energy consistent, and strictly positivity preserving. Such a scheme should not only remain robust in the singular relaxation regime, but also faithfully reproduce the key structural properties of both the relaxation system and its limiting equation. The purpose of the present work is therefore to exploit the lower-order formulation \eqref{hyperbolic_system} at the fully discrete level and construct a numerical method that accurately captures the behavior of the solution as the relaxation parameter tends to zero. In particular, we employ an IMEX time discretization in which the stiff relaxation terms are treated implicitly, while the remaining nonlinear transport and coupling terms are treated explicitly.

A central difficulty is to establish a discrete energy inequality for the fully discrete scheme. To this end, suitable numerical viscosity terms are introduced into the discretization; see Theorem \ref{thm:eng-stable}.  Our construction follows an approach similar in spirit to that of \cite{berthon2023artificial}, where artificial viscosity is incorporated in a manner compatible with the underlying energy structure. While the addition of artificial viscosity in all equations can readily ensure stability, at the same time they will be implicit in time. This may lead to a computationally expensive coupled system. A key challenge is therefore to identify the minimal set of diffusive terms needed to establish the discrete energy stability, while retaining an efficient semi-implicit formulation and the AP property. We therefore introduce diffusion only where necessary for the discrete energy estimate, thereby reducing the implicitness while retaining energy stability and the AP property. Moreover, because the stiffness is present in both the source and flux contributions, part of the considered numerical viscosity is still implicit in time. However, we retain a linearly implicit structure such that the scheme can be implemented without Newton-type nonlinear iterations.

The resulting method is AP and, in the limit of vanishing relaxation parameter, reduces to a consistent discretization of the limiting fourth-order equation without requiring the numerical mesh to resolve the small relaxation scale. At the same time, the scheme preserves the other key structural properties of the continuous model, including the conservation of relevant physical quantities, the discrete energy inequality, and positivity of the film height. These properties are essential for obtaining robust numerical approximations of the underlying degenerate fourth-order thin-film dynamics.

To make the following discussions precise, we now state the explicit assumptions on the mobility function $\mathcal{M}(u)$ and the pressure function $\Pi(u)$. They are assumed to satisfy the following conditions.
%\medskip\\
\begin{assumption}\label{assumption.A}
\leavevmode
Let $L, m_0, \delta>0$ be constants. We have
\begin{enumerate}[label=($\mathcal{A}.\arabic*$),ref=$\mathcal{A}.\arabic*$]
  \item\label{A.1}The mobility function $\mathcal{M}:\mathbb{R}\to\mathbb{R}_+$ satisfies $\mathcal{M}\in C^\infty(\mathbb{R})$ with $\max_{v\in \mathbb{R}}\,\lvert \mathcal{M}'(v)\rvert < L$ and $\mathcal{M}(v)>1/m_0>0$ for all $v\in [0, \infty)$.
  \item\label{A.2} The pressure function $\Pi:\mathbb{R}\to\mathbb{R}$ satisfies $\Pi\in C^\infty(\mathbb{R})$ with $\max_{v\in \mathbb{R}}\lvert \Pi'(v)\rvert,\,\, \max_{v\in \mathbb{R}}\lvert \Pi''(v)\rvert  < L$ and $\Pi'(v)\geq \delta$ for all $v\in [0, \infty)$.
\end{enumerate}
\end{assumption}

The rest of the article is organized as follows. In section \ref{sec:discrete_spaces} we define the discrete notations and provide the basic discrete identities required in the article. Section \ref{sec:structure-preserving_scheme} is devoted to developing the numerical scheme for the system \eqref{hyperbolic_system} and proving its structure-preserving properties. In Section \ref{sec:numerics}, we provide a series of numerical test cases which validate the accuracy and robustness of the numerical scheme. Conclusions and future outlooks are provided in Section \ref{sec: conclusions}.%%%%%write down the structure here.

\section{Notations, discrete spaces and basic discrete identities}\label{sec:discrete_spaces}
In this section, we collect some basic discrete identities for forward, backward, and central differences under periodic boundary conditions that are utilized in the following sections. Let $\Omega=\{(a,b)^d: a, b\in\mathbb{R}, a<b, d\in \mathbb{N}\}$ be the computational domain, and let $\Delta x=(b-a)/N$ denote the step size, which is considered to be uniform in each spatial direction. More precisely, we consider the uniform Cartesian mesh of the form
\[
\mathbf{x}_\mathbf{j}=\mathbf{j}\Delta x,\qquad \mathbf{j}=(j_1,\ldots,j_d),
\qquad 0\leq j_\alpha\leq N-1,\quad \alpha=1,\ldots,d.
\]
We denote by $\mathbf{e}_\alpha$ the unit multi-index in the $\alpha$-th coordinate direction. A scalar grid function is denoted by
$f=\{f_{\mathbf j}\}_{\mathbf j}$, while a vector-valued grid function is denoted by
$\mathbf v=\{\mathbf v_{\mathbf j}\}_{\mathbf j}$, with
\[
\mathbf v_{\mathbf j}=(v_{\mathbf{j}}^1,\ldots,v_{\mathbf{j}}^d).
\]
For a scalar grid function $f$, we define the forward and backward difference operators in the $\alpha$-th coordinate direction by
\[
(D_\alpha^+ f)_{\mathbf j}=\frac{f_{\mathbf j+\mathbf e_\alpha}-f_{\mathbf j}}{\Delta x},
\qquad
(D_\alpha^- f)_{\mathbf j}=\frac{f_{\mathbf j}-f_{\mathbf j-\mathbf e_\alpha}}{\Delta x}.
\]
The discrete gradient and divergence are then defined as
\[
(\nabla_h^+ f)_{\mathbf j}=\bigl((D_1^+f)_{\mathbf j},\ldots,(D_d^+f)_{\mathbf j}\bigr),
\]
and
\[
(\nabla_h^-\cdot\mathbf v)_{\mathbf j}=\sum_{\alpha=1}^d (D_\alpha^-v^\alpha)_{\mathbf j}.
\]
Moreover, we define the centered second-order discrete Laplacian as
\[
(\Delta_h f)_{\mathbf{j}}:=(\nabla_h^-\cdot\nabla_h^+f)_{\mathbf{j}}
=
\sum_{\alpha=1}^d (D_\alpha^-D_\alpha^+f)_{\mathbf{j}},
\]
where for vector-valued grid functions, $\Delta_h\mathbf v$ is understood componentwise, i.e.
\[
(\Delta_h\mathbf v)_{\mathbf j}=\bigl((\Delta_h v^1)_{\mathbf j},\ldots,(\Delta_h v^d)_{\mathbf j}\bigr),
\]
and also $\nabla_h^+$ and $\nabla_h^-$:

\[
(\nabla_h^{\pm}\mathbf v)_{\mathbf{j}}
:=
\big((D_\alpha^{\pm}v^\beta)_{\mathbf{j}}\big)_{\beta,\alpha=1}^d,
\]
In particular, for a two-dimensional Cartesian mesh, we write
\[
\mathbf j=(i,k),\qquad\mathbf x_{\mathbf j}=(x_i,y_k)=(i\Delta x,k\Delta x).
\]
Then
\[
(D_1^-D_1^+f)_{i,k}=\frac{f_{i+1,k}-2f_{i,k}+f_{i-1,k}}{\Delta x^2},
\]
and
\[
(D_2^-D_2^+f)_{i,k}=\frac{f_{i,k+1}-2f_{i,k}+f_{i,k-1}}{\Delta x^2}.
\]
Therefore,
\[
\begin{aligned}
(\Delta_h f)_{i,k}&=(D_1^-D_1^+f)_{i,k}+(D_2^-D_2^+f)_{i,k} \\
&=\frac{f_{i+1,k}-2f_{i,k}+f_{i-1,k}}{\Delta x^2}+\frac{f_{i,k+1}-2f_{i,k}+f_{i,k-1}}{\Delta x^2}.
\end{aligned}
\]
Equivalently,
\[
(\Delta_h f)_{i,k}=\frac{f_{i+1,k}+f_{i-1,k}+f_{i,k+1}+f_{i,k-1}-4f_{i,k}}{\Delta x^2},
\]
which is the classical five-point discrete Laplacian.

We denote by $L_{\Delta}^2$ the space of scalar grid functions and by $L_{\Delta,0}^2$ the subspace of grid functions satisfying homogeneous or periodic boundary conditions on $\partial\Omega$. The discrete scalar product for scalar grid functions $f$ and $g$ is defined as 
\[
{\langle f,g\rangle}_\Delta:=\Delta x^d\sum_{\mathbf j} f_{\mathbf j}g_{\mathbf j},
\]
with corresponding norm
\[
{\|f\|}_{2, \Delta}:=\left(\Delta x^d\sum_{\mathbf j}|f_{\mathbf j}|^2\right)^{1/2}.
\]
For vector-valued grid functions $\mathbf u$ and $\mathbf v$, we define the discrete inner product
\[
{\langle \mathbf u,\mathbf v\rangle}_\Delta:=\Delta x^d\sum_{\mathbf j}\mathbf u_{\mathbf j}\cdot\mathbf v_{\mathbf j},
\]
and the associated discrete $L^2$ norm as
\[
{\|\mathbf v\|}_{2, \Delta}:=\left(\Delta x^d\sum_{\mathbf j}|\mathbf v_{\mathbf j}|^2\right)^{1/2}.
\]
We also use the discrete $L^\infty$-norm
\[
{\|f\|}_{\infty,\Delta}:=\max_{\mathbf j}|f_{\mathbf j}|.
\]
The following identities are the discrete analogues of integration by parts. They hold under periodic boundary conditions, or under homogeneous boundary conditions for which the corresponding boundary terms vanish. We write some of them with proof and some which are without proof can be easily derived from classical Abel’s Lemma on summation by parts (SBP) (for example, see \cite{chu2007abel}). 
\begin{lemma}[Discrete integration by parts (SBP property)]
\label{lemma:discrete_ibp}
For scalar grid functions $f,g\in L_{\Delta,0}^2$ and vector-valued grid functions $\mathbf v\in (L_{\Delta,0}^2)^d$, the following discrete 
\[
{\langle \nabla_h^+ f,\mathbf v\rangle}_\Delta
=
-{\langle f,\nabla_h^-\cdot\mathbf v\rangle}_\Delta.
\]
Equivalently,
\[
{\langle\nabla_h^-\cdot\mathbf v,f\rangle}_\Delta
=
-{\langle\mathbf v,\nabla_h^+ f\rangle}_\Delta.
\]
In particular,
\[
{\langle\Delta_h f,f\rangle}_\Delta
=
-{\|\nabla_h^+f\|}_{2, \Delta}^2.
\]
For vector-valued grid functions,
\[
{\langle\Delta_h\mathbf v,\mathbf v\rangle}_\Delta
=
-{\|\nabla_h^+\mathbf v\|}_{2, \Delta}^2.
\]
\end{lemma}
In what follows, we also make use of the following elementary discrete estimate. For any vector-valued grid functions $\mathbf v\in (L_{\Delta,0}^2)^d$, we have
\begin{equation}
\label{eq:div_less_grad}
    \|\nabla_h^-\cdot\mathbf v\|_{2, \Delta}^2
\le d\,\|\nabla_h^\pm\mathbf v\|_{2, \Delta}^2,
\end{equation}
which follows directly from the Cauchy--Schwarz inequality,
\begin{equation}
\left|\sum_{i=1}^d a_i\right|^2
\le d\sum_{i=1}^d |a_i|^2,
\end{equation}
where $a_i\in\mathbb{R}$, for all $i\in\{1,...,d\}$.

We further make use of the following discrete inverse estimates for periodic boundary conditions:
\begin{lemma}[Discrete inverse estimate]
\label{lemma:discrete_inverse}
There exists a constant $C_d>0$, depending only on the space dimension $d$ and on the choice of finite difference operators, such that
\begin{equation}
    \label{eq:scalar_lap-grad}
    {\|\Delta_h f\|}_{2, \Delta}^2
\leq
\frac{C_d}{\Delta x^2}
{\|\nabla_h^+ f\|}_{2, \Delta}^2.
\end{equation}
For the standard Cartesian discretization defined above, one may take $C_d=4d$. The same estimate holds componentwise for vector-valued grid functions:
\begin{equation}
    \label{eq:vector_lap-grad}
    {\|\Delta_h\mathbf v\|}_{2, \Delta}^2
\leq
\frac{C_d}{\Delta x^2}
{\|\nabla_h^+\mathbf v\|}_{2, \Delta}^2.
\end{equation}
\end{lemma}
\begin{proof}
Here we prove this Lemma for $d = 1$ for the sake of simplicity, using
\[
(\Delta_h f)_i=\frac{(\nabla_h^+f)_i-(\nabla_h^+f)_{i-1}}{\Delta x},
\]
and the following standard inequality
\[
(a-b)^2\le 2(a^2+b^2),
\]
we obtain
\begin{align*}
\|\Delta_h f\|_{2,\Delta}^2
&=
\frac{1}{\Delta x^2}\sum_{i=0}^L
\Delta x((\nabla_h^+f)_i-(\nabla_h^+f)_{i-1})^2\\
&\le\frac{2}{\Delta x^2}\sum_{i=0}^{L-1}
\Delta x(\nabla_h^+f)_i^2
+\frac{1}{\Delta x^2}\Delta x (\nabla_h^+f)_{-1}^2+\frac{1}{\Delta x^2}\Delta x(\nabla_h^+f)_L^2.
\end{align*}
Thanks to periodic boundary, we further get
\begin{equation}
\|\Delta_h f\|_{2,\Delta}^2\le\frac{4}{\Delta x^2}\|\nabla_h^+f\|_{2,\Delta}^2.
\end{equation}
In this case, the estimate \eqref{eq:scalar_lap-grad} holds with $C_d=4$ and a similar procedure can be adapted to prove the estimate for $d>1$. Also the estimate \eqref{eq:vector_lap-grad} for the vector function $\mathbf{v}$ is now straightforward as all its components satisfy \eqref{eq:scalar_lap-grad}.
\end{proof}

\section{An energy stable semi-implicit scheme for the system \eqref{hyperbolic_system}}
\label{sec:structure-preserving_scheme}
In this section, we introduce a novel implicit--explicit scheme for the system \eqref{hyperbolic_system}. The scheme is constructed so that the numerical solution satisfies an asymptotic-preserving property, while the discrete energy remains stable and consistent with its continuous counterpart. We adopt a mixed implicit-explicit strategy motivated by AP time splitting, 
which was introduced in Jin et al. \cite{jin1998diffusive}. The stiff terms in \eqref{hyperbolic_system} are treated implicitly in the predictor stage, producing intermediate values, which are then updated through an explicit corrector step. In this sense, the method can also be interpreted as a predictor-corrector implicit-explicit scheme.
\subsection{Numerical scheme and existence of numerical solution}\label{sec: scheme}
Given the numerical solution $\mathbf{U}_{\mathbf{j}}^n=(u_{\mathbf{j}}^n,\psi_{\mathbf{j}}^n,\mathbf{q}_{\mathbf{j}}^n,w_{\mathbf{j}}^n,\mathbf{p}_{\mathbf{j}}^n)$ at time $t^n$, the predictor step computes the intermediate state $\mathbf{U}_{\mathbf{j}}^*=(u_{\mathbf{j}}^{*},\psi^*_{\mathbf{j}},\mathbf{q}_{\mathbf{j}}^{*},w_{\mathbf{j}}^{*},\mathbf{p}_{\mathbf{j}}^{*})$ as follows.\medskip \\
\textbf{Implicit step:}
\begin{subequations}
\label{eq:semi-imp}
\begin{align}
\frac{u_{\mathbf{j}}^{*} - u_{\mathbf{j}}^n}{\Delta t} &= 0, 
\label{eq:h_updt_star}\\
\frac{\psi^*_{\mathbf{j}} - \psi_{\mathbf{j}}^n}{\Delta t} 
&= -\frac{1}{\epsilon}
\left(
w^*_{\mathbf{j}}+(1-\epsilon)(\nabla_h^- \cdot \mathbf{q}^{*})_{\mathbf{j}}
\right),
\label{eq:psi_updt_star}\\
\frac{\mathbf{q}_{j}^{*}-\mathbf{q}_{j}^{n}}{\Delta t} 
&= -\frac{1}{\epsilon}
\left(
\frac{\mathbf{q}_{j}^{*}}{\mathcal{M}_{\mathbf{j}}^n}
+(1-\epsilon)\nabla_h^+ \bigl(\Pi^*+\psi^{*}\bigr)_{\mathbf{j}}
-r(\Delta_{h} \mathbf{q}^{*})_{\mathbf{j}}
\right),
\label{eq:q_updt_star}\\
\frac{w_{\mathbf{j}}^{*} - w_{\mathbf{j}}^{\,n}}{\Delta t} 
&= \frac{1}{\epsilon}
\left(
\psi^*_{\mathbf{j}}
+\gamma(1-\epsilon)(\nabla_h^- \cdot \mathbf{p}^{*})_{\mathbf{j}}
+s(\Delta_{h}w^{*})_{\mathbf{j}}
\right),
\label{eq:w_updt_star}\\
\frac{\mathbf{p}_{\mathbf{j}}^{*} - \mathbf{p}_{\mathbf{j}}^n}{\Delta t} &= 0.
\label{eq:p_updt_star}
\end{align}
\end{subequations}
In Lemma \ref{lemma: num_existence}, we prove that the predictor system \eqref{eq:semi-imp} is well-posed and admits a unique solution $\mathbf{U}_{\mathbf{j}}^*.$
The intermediate state is then used to compute the updated solution $\mathbf{U}_{\mathbf{j}}^{n+1}$ through the following corrector step.\medskip\\
\textbf{Explicit step:}
\begin{subequations}
\label{eq:expl}
\begin{align}
&\frac{u_{\mathbf{j}}^{n+1} - u_{\mathbf{j}}^{*}}{\Delta t} + (\nabla_h^- \cdot\mathbf{q}^*)_{\mathbf{j}} = r_1(\Delta_{h} \Pi^*)_{\mathbf{j}}, \label{eq:h_updt}\\
&\frac{\psi_{\mathbf{j}}^{\,n+1} - \psi^*_{\mathbf{j}}}{\Delta t} + (\nabla_h^- \cdot\mathbf{q}^*)_{\mathbf{j}} = r_2(\Delta_h \psi^*)_{\mathbf{j}}, \label{eq:psi_updt}\\
&\frac{\mathbf{q}_{j}
^{\,n+1}-\mathbf{q}_{j}^{*}}{\Delta t} + \nabla_h^+(\Pi^*+\psi^{*})_{\mathbf{j}} = 0, \label{eq:q_updt}\\
&\frac{w_{\mathbf{j}}^{\,n+1} - w_{\mathbf{j}}^{*}}{\Delta t} -\gamma (\nabla_h^- \cdot \mathbf{p}^*)_{\mathbf{j}} = 0, \label{eq:w_updt}\\
&\frac{\mathbf{p}_{\mathbf{j}}^{\,n+1} - \mathbf{p}_{\mathbf{j}}^{*}}{\Delta t} - (\nabla_h^+ w^{*})_{\mathbf{j}}= r_3(\Delta_h \mathbf{p}^*)_{\mathbf{j}}.\label{eq:p_updt}
\end{align}
\end{subequations}
where we use the notation $\Pi^n_{\mathbf{j}}=\Pi(u_{\mathbf{j}}^n)$,  $\Pi^*_{\mathbf{j}}=\Pi(u_{\mathbf{j}}^*)$ and $\mathcal{M}_{\mathbf{j}}^n:=\mathcal{M}(u_{\mathbf{j}}^n)$. In \eqref{eq:semi-imp}-\eqref{eq:expl}, $r, s, r_1, r_2, r_3$ are suitably chosen positive numbers. In particular, in order to ensure energy stability (see Theorem \ref{thm:eng-stable}), we choose 
\[
r=c_\mathbf{q} \Delta x, \quad s=c_w \Delta x, \quad r_1=c_u \epsilon\Delta t, \quad r_2=c_\psi \Delta t, \quad r_3=c_\mathbf{p}\epsilon d\gamma\Delta t,
\]
where $c_\mathbf{q}, c_w>0$ and $c_u, c_\psi, c_\mathbf{p}>1$ are suitably choosen constants. In the proof of Theorem  \ref{thm:eng-stable}, we choose in particular $c_\mathbf{q}=c_w=1$ and $c_u=c_\psi= c_\mathbf{p}=2$.

Before presenting some notable properties of the scheme \eqref{eq:semi-imp}-\eqref{eq:expl}, we introduce an important notation for well-prepared initial data. We assume throughout the paper that our initial data are well-prepared. This is mainly crucial to show the AP property and the energy stability of the numerical scheme. We define the well-prepared initial data as follows.
\begin{definition}[Well-prepared initial data]
\label{def:well-prepared_init}
A family of initial data $U_0^\epsilon
=
\left(u_0^\epsilon,\psi_0^\epsilon,\mathbf q_0^\epsilon,
w_0^\epsilon,\mathbf p_0^\epsilon\right)$ of \eqref{hyperbolic_system} is called well-prepared if
\begin{equation}\label{initial_data_hyp_relaxation}
\begin{aligned}
\nabla\cdot\mathbf q_0^\epsilon+w_0^\epsilon&=O(\epsilon),\\
\mathbf q_0^\epsilon+
\mathcal M(u_0^\epsilon)
\nabla\!\left(\Pi(u_0^\epsilon)+\psi_0^\epsilon\right)
&=O(\epsilon),\\
\psi_0^\epsilon+\gamma\nabla\cdot\mathbf p_0^\epsilon&=O(\epsilon),\\
\mathbf{p}_0^\epsilon=\nabla u_0^\epsilon&+O(\epsilon),
\end{aligned}
\end{equation}
and the initial energy is uniformly bounded, i.e.\ $E^\epsilon(0)\le C$, for some constant $C$ independent of $\epsilon$.
\end{definition}

First, we proceed to prove that the proposed scheme \eqref{eq:semi-imp}-\eqref{eq:expl} possesses a numerical solution and thus can be solved from time step $t=t^n$ to $t=t^{n+1}$. Since the corrector step only depends on the predictor values $\mathbf{U}_{\mathbf{j}}^*$, it is then sufficient to prove that the predictor step produces a unique numerical solution. We prove this in the following lemma, assuming that the discrete mobility function $\mathcal{M}_{\mathbf{j}}^n>0$ for all $n\geq 0$.
\begin{lemma}[Existence and uniqueness of the numerical solution]\label{lemma: num_existence}
Let $\mathbf{U}_{\mathbf{j}}^n$ be given and assume that the discrete mobility satisfies $\mathcal{M}_{\mathbf{j}}^n > 0$ for all $\mathbf{j}$ and $n\geq 0$. Then the numerical scheme as defined in \eqref{eq:semi-imp}-\eqref{eq:expl} admits a unique solution $\mathbf{U}_{\mathbf{j}}^{n+1}$.
\end{lemma}
\begin{proof}
Since $u^*=u^n$ and $\mathbf{p}_{\mathbf{j}}^{*}=\mathbf{p}_{\mathbf{j}}^{n}$ are known from the first equation of \eqref{eq:semi-imp}, the quantity $\Pi^*=\Pi(u^*)$ is also known in the predictor step. Hence, the remaining equations form a linear system for $(\psi^*,\mathbf{q}^*,w^*)$. It is therefore enough to show that the associated homogeneous system has only the trivial solution. Therefore, we consider the scheme  \eqref{eq:semi-imp} with zero right-hand side, which leads to
\begin{subequations}
\label{eq:exist_system}
\begin{align}
\frac{u_{\mathbf{j}}^{*}}{\Delta t} &= 0, \label{eq:h_null}\\
\frac{\psi^*_{\mathbf{j}}}{\Delta t} &= -\frac{1}{\epsilon}\left(w_{\mathbf{j}}^*+(1-\epsilon)(\nabla_h^- \cdot \mathbf{q}^*)_{\mathbf{j}}\right), \label{eq:psi_null}\\
\frac{\mathbf{q}_{j}^{*}}{\Delta t} &= -\frac{1}{\epsilon} \left(\frac{\mathbf{q}_{\mathbf{j}}^*}{\mathcal{M}_{\mathbf{j}}^n}+(1-\epsilon)(\nabla_h^+\psi^*)_{\mathbf{j}}-r(\Delta_h\mathbf{q}^*)_{\mathbf{j}}\right),\label{eq:q_null}\\
\frac{w_{\mathbf{j}}^{*}}{\Delta t} 
&= \frac{1}{\epsilon}\left(\psi^*_{\mathbf{j}}+s(\Delta_h w^*)_{\mathbf{j}}\right),\label{eq:w_null}\\
\frac{\mathbf{p}_{\mathbf{j}}^{*}}{\Delta t} &= 0.\label{eq:p_null}
\end{align}
\end{subequations}
From \eqref{eq:h_null} and \eqref{eq:p_null} we immediately obtain
\begin{equation}
   u^* = 0, \quad \mathbf{p}^* = 0. 
\end{equation}
The equation \eqref{eq:psi_null} for $\psi^*$ becomes
\begin{equation}
\psi^*=-\frac{\Delta t}{\epsilon}\left(w^*+(1-\epsilon)\nabla_h^-\cdot \mathbf{q}^*\right).
\end{equation}
Substituting this expression into the equations for $\mathbf{q}^*$ and $w^*$, we obtain a reduced system in $(\mathbf{q}^*, w^*)$. We now take the discrete inner product of the $\mathbf{q}^*$-equation \eqref{eq:q_null} with $\mathbf{q}^*$ and sum over $j$ to obtain 
\begin{align}
\frac{1}{\Delta t}\|\mathbf{q}^*\|^2
&= -\frac{1}{\epsilon}\left(
\sum_{\mathbf{j}} \frac{|\mathbf{q}_{\mathbf{j}}^*|^2}{\mathcal{M}_{\mathbf{j}}^n}
+ (1-\epsilon)\langle \nabla_h^+ \psi^*, \mathbf{q}^* \rangle
- r \langle \Delta_h \mathbf{q}^*, \mathbf{q}^* \rangle
\right).
\end{align}
Using the SBP property from Lemma \ref{lemma:discrete_ibp}, we have
\[
-\langle \Delta_h \mathbf{q}^*,\mathbf{q}^*\rangle
=
\|\nabla_h^+\mathbf{q}^*\|^2, \quad \langle \nabla_h^+\psi^*,\mathbf{q}^*\rangle
=
-\langle \psi^*,\nabla_h^-\cdot\mathbf{q}^*\rangle.
\]
Therefore, we obtain
\begin{align}\label{eq: q_update}
\frac{1}{\Delta t}\|\mathbf{q}^*\|^2
+
\frac{1}{\epsilon}
\sum_{\mathbf{j}}
\frac{|\mathbf{q}_{\mathbf{j}}^*|^2}{\mathcal{M}_{\mathbf{j}}^n}
+
\frac{r}{\epsilon}
\|\nabla_h^+\mathbf{q}^*\|^2
=
\frac{1-\epsilon}{\epsilon}
\langle \psi^*,\nabla_h^-\cdot\mathbf{q}^*\rangle.
\end{align}
Similarly, multiplying the $w^*$-equation \eqref{eq:w_null} by $w^*$ and summing together with SBP property from Lemma \ref{lemma:discrete_ibp} yields
\begin{align}\label{w_update}
\frac{1}{\Delta t}\|w^*\|^2
+
\frac{s}{\epsilon}\|\nabla_h^+ w^*\|^2
=
\frac{1}{\epsilon}
\langle \psi^*,w^*\rangle.
\end{align}
Thus, adding the two identities \eqref{eq: q_update} and \eqref{w_update} gives
\begin{align}\label{q_update}
\frac{1}{\Delta t}\left(\|\mathbf{q}^*\|^2+\|w^*\|^2\right)+\frac{1}{\epsilon}
\sum_{\mathbf{j}}\frac{|\mathbf{q}_{\mathbf{j}}^*|^2}{\mathcal{M}_{\mathbf{j}}^n} +\frac{r}{\epsilon}\|\nabla_h^+\mathbf{q}^*\|^2+\frac{s}{\epsilon}\|\nabla_h^+w^*\|^2=\frac{1}{\epsilon}\left\langle\psi^*,w^*+(1-\epsilon)\nabla_h^-\cdot\mathbf{q}^*\right\rangle.
\end{align}
Substituting the expression for $\psi^*$ from \eqref{eq:psi_null} in \eqref{q_update}, we obtain the following equality
\begin{align}
\frac{1}{\Delta t}\left(\|\mathbf{q}^*\|^2+\|w^*\|^2+\|\psi^*\|^2\right) +\frac{1}{\epsilon}
\sum_{\mathbf{j}}\frac{|\mathbf{q}_{\mathbf{j}}^*|^2}{\mathcal{M}_{\mathbf{j}}^n} +\frac{r}{\epsilon}\|\nabla_h^+\mathbf{q}^*\|^2+\frac{s}{\epsilon}\|\nabla_h^+w^*\|^2=0.
\end{align}
Since $\Delta t>0$, $\epsilon>0$, $\mathcal{M}_{\mathbf{j}}^n>0$, and $r,s\geq 0$, all terms on the left-hand side are nonnegative. Hence,
\[
\mathbf{q}^*=0,\qquad w^*=0,\qquad  \psi^*=0.
\]
Together with $u^*=0$ and $\mathbf{p}^*=0$, this shows that the homogeneous system has only the trivial solution. Thus, the kernel of the associated linear operator is a null set. Since the system is finite-dimensional, the operator is invertible. Therefore, the predictor step \eqref{eq:semi-imp} admits a unique solution. Since the corrector step \eqref{eq:expl} is explicit once $\mathbf{U}^*$ is known, the full scheme \eqref{eq:semi-imp}-\eqref{eq:expl} also admits a unique numerical solution.
\end{proof}

\subsection{Asymptotic preserving property of the numerical scheme \eqref{eq:semi-imp}-\eqref{eq:expl}}
\label{sec:AP}

In this section, we formally derive the relaxation limit of the scheme \eqref{eq:semi-imp}-\eqref{eq:expl}. We show that the limit $\epsilon\to 0$ of the numerical scheme \eqref{eq:semi-imp}-\eqref{eq:expl} is a consistent numerical scheme for the limit equation \eqref{eq: main}.

\begin{theorem}[AP property]
Assume that discrete solutions $\{\mathbf{U}_{\mathbf{j}}^n = (u_{\mathbf{j}}^n, \psi_{\mathbf{j}}^n, \mathbf{q}_{\mathbf{j}}^n, w_{\mathbf{j}}^n, \mathbf{p}_{\mathbf{j}}^n)\}_{\mathbf{j},n}$ obtained from the numerical scheme \eqref{eq:semi-imp}-\eqref{eq:expl} has the following asymptotic expansion
\begin{equation}
\label{eq:asymp-exp}
\mathbf{U}_{\mathbf{j}}^n = \mathbf{U}_{\mathbf{j}}^{n,(0)} + \epsilon \mathbf{U}_{\mathbf{j}}^{n,(1)} + \epsilon^2 \mathbf{U}_{\mathbf{j}}^{n,(2)} + \cdots,
\end{equation}
where
\begin{equation}
    \mathbf{U}_{\mathbf{j}}^{n,(k)} = \big(u_{\mathbf{j}}^{n,(k)}, \psi_{\mathbf{j}}^{n,(k)}, \mathbf{q}_{\mathbf{j}}^{n,(k)}, w_{\mathbf{j}}^{n,(k)}, \mathbf{p}_{\mathbf{j}}^{n,(k)}\big),
\quad k=0,1,2,\dots,
\end{equation}
for all $\mathbf{j}, n$.

Furthermore, assume that the numerical solutions at the $n$-th step, $\mathbf{U}_{\mathbf{j}}^n = (u_{\mathbf{j}}^n, \psi_{\mathbf{j}}^n, \mathbf{q}_{\mathbf{j}}^n, w_{\mathbf{j}}^n, \mathbf{p}_{\mathbf{j}}^n)$, are well-prepared in the sense of Definition \eqref{def:well-prepared_init}. Then, the numerical scheme \eqref{eq:semi-imp}-\eqref{eq:expl} becomes a consistent numerical approximation of \eqref{eq: main} in the relaxation limit (as $\epsilon\to 0$).
\end{theorem}
\begin{proof}
We start by adding \eqref{eq:semi-imp} and \eqref{eq:expl} and obtain
\begin{subequations}
\label{eq:scheme_tog}
\begin{align}
&\frac{u_{\mathbf{j}}^{n+1} - u_{\mathbf{j}}^{n}}{\Delta t} + (\nabla_h^- \cdot\mathbf{q}^*)_{\mathbf{j}} = r_1(\Delta_{h} \Pi^*)_{\mathbf{j}}, \\
&\frac{\psi_{\mathbf{j}}^{n+1} - \psi_{\mathbf{j}}^n}{\Delta t} 
= -\frac{1}{\epsilon}
\left(
w^*_{\mathbf{j}}+(\nabla_h^- \cdot \mathbf{q}^{*})_{\mathbf{j}}\right)+r_2(\Delta_h \psi^*)_{\mathbf{j}},\\
&\frac{\mathbf{q}_{j}^{n+1}-\mathbf{q}_{j}^{n}}{\Delta t} 
= -\frac{1}{\epsilon}
\left(
\frac{\mathbf{q}_{j}^{*}}{\mathcal{M}_{\mathbf{j}}^n}
+\nabla_h^+ \bigl(\Pi^*+\psi^{*}\bigr)_{\mathbf{j}}
-r(\Delta_{h} \mathbf{q}^{*})_{\mathbf{j}}
\right),\\
&\frac{w_{\mathbf{j}}^{n+1} - w_{\mathbf{j}}^{\,n}}{\Delta t} 
= \frac{1}{\epsilon}
\left(
\psi^*_{\mathbf{j}}
+\gamma(\nabla_h^- \cdot \mathbf{p}^{*})_{\mathbf{j}}
+s(\Delta_{h}w^{*})_{\mathbf{j}}
\right),\\
&\frac{\mathbf{p}_{\mathbf{j}}^{\,n+1} - \mathbf{p}_{\mathbf{j}}^{n}}{\Delta t} - (\nabla_h^+ w^{*})_{\mathbf{j}}= r_3(\Delta_h \mathbf{p}^*)_{\mathbf{j}}.
\end{align}
\end{subequations}
Then substituting the asymptotic expansion \eqref{eq:asymp-exp} into the numerical scheme \eqref{eq:scheme_tog} yields
\begin{subequations}
\label{eq:limit_scheme_tog}
\begin{align}
&\frac{u_{\mathbf{j}}^{n+1,(0)} - u_{\mathbf{j}}^{n,(0)}}{\Delta t} + (\nabla_h^- \cdot\mathbf{q}^*)_{\mathbf{j}} = r_1(\Delta_{h} \Pi^*)_{\mathbf{j}}+\mathcal{O}(\epsilon),\label{eq:h_updt_tog}\\
&\mathcal{O}(\epsilon) 
=
w^*_{\mathbf{j}}+(\nabla_h^- \cdot \mathbf{q}^{*})_{\mathbf{j}},\label{eq:psi_updt_tog}\\
&\mathcal{O}(\epsilon) 
= 
\frac{\mathbf{q}_{j}^{*}}{\mathcal{M}_{\mathbf{j}}^{n,(0)}}
+\nabla_h^+ \bigl(\Pi^{n,(0)}+\psi^{*}\bigr)_{\mathbf{j}}
-r(\Delta_{h} \mathbf{q}^{*})_{\mathbf{j}},\label{eq:q_updt_tog}\\
&\mathcal{O}(\epsilon)=
\psi^*_{\mathbf{j}}
+\gamma(\nabla_h^- \cdot \mathbf{p}^{n,(0)})_{\mathbf{j}}
+s(\Delta_{h}w^{*})_{\mathbf{j}},\label{eq:w_updt_tog}\\
&\frac{\mathbf{p}_{\mathbf{j}}^{\,n+1,(0)} - \mathbf{p}_{\mathbf{j}}^{n,(0)}}{\Delta t} - (\nabla_h^+ w^{*})_{\mathbf{j}}= r_3(\Delta_h \mathbf{p}^{n,(0)})_{\mathbf{j}}+\mathcal{O}(\epsilon).\label{eq:p_updt_tog}
\end{align}
\end{subequations}
Using \eqref{eq:psi_updt_tog}, \eqref{eq:q_updt_tog} and \eqref{eq:w_updt_tog}, we also have
\begin{equation}
\label{eq:q_updt_tog-lim}
\frac{\mathbf{q}_{j}^{*}}{\mathcal{M}_{\mathbf{j}}^{n,(0)}}= 
-\nabla_h^+ \bigl(\Pi^{n,(0)}-\gamma(\nabla_h^- \cdot \mathbf{p}^{n,(0)})
-s(\Delta_{h}w^{*})\bigr)_{\mathbf{j}}+r(\Delta_{h} \mathbf{q}^{*})_{\mathbf{j}}+\mathcal{O}(\epsilon).
\end{equation}
In what follows, we show that the scheme \eqref{eq:limit_scheme_tog} is indeed an approximation of the limit equation \eqref{eq: main}.

To be more precise we write $\mathbf{U}^*_\mathbf{j}=\mathbf{U}^{*,n}_\mathbf{j}$ for all $\mathbf{j}$ at this stage and with this the scheme \eqref{eq:limit_scheme_tog} reduces to
\begin{subequations}
\label{eq:limit_scheme_tog1}
\begin{align}
&\frac{u_{\mathbf{j}}^{n+1,(0)} - u_{\mathbf{j}}^{n,(0)}}{\Delta t} + (\nabla_h^- \cdot\mathbf{q}^{*,n})_{\mathbf{j}} = r_1(\Delta_{h} \Pi^{*,n})_{\mathbf{j}}+\mathcal{O}(\epsilon),\label{eq:h_updt_tog1}\\
&\frac{\mathbf{p}_{\mathbf{j}}^{\,n+1,(0)} - \mathbf{p}_{\mathbf{j}}^{n,(0)}}{\Delta t} - \bigg(\nabla_h^+ \frac{u^{n+1,(0)} - u^{n,(0)}}{\Delta t}\bigg)_{\mathbf{j}}= r_3(\Delta_h \mathbf{p}^{n,(0)})_{\mathbf{j}}-r_1(\nabla_h^+(\Delta_{h} \Pi^{*,n}))_{\mathbf{j}}+\mathcal{O}(\epsilon),\label{eq:p_updt_tog1}
\end{align}
\end{subequations}
and using \eqref{eq:psi_updt_tog} and \eqref{eq:q_updt_tog-lim}, we also have
\begin{equation}
\label{eq:q_updt_tog-lim1}
\frac{\mathbf{q}_{j}^{{*},n}}{\mathcal{M}_{\mathbf{j}}^{n,(0)}}= 
-\nabla_h^+ \bigl(\Pi^{n,(0)}-\gamma(\nabla_h^- \cdot \mathbf{p}^{n,(0)})
+s(\Delta_{h}\nabla_h^- \cdot \mathbf{q}^{*,n})\bigr)_{\mathbf{j}}+r(\Delta_{h} \mathbf{q}^{*,n})_{\mathbf{j}}+\mathcal{O}(\epsilon).
\end{equation}

Assuming all the higher derivatives of the discrete unknowns remain $\mathcal{O}(1)$ as in \cite{anandan2024high, anandan2025asymptotic}, one further obtains from \eqref{eq:q_updt_tog-lim1} the identity.
\begin{equation}
\begin{aligned}
\label{eq:q_updt_tog-lim2}
\mathbf{q}^{{*},n}=(I-r\mathcal{M}^{n,(0)}\Delta_h+s\mathcal{M}^{n,(0)}\nabla_h^+\Delta_{h}\nabla_h^-\cdot)^{-1}\{ 
-\mathcal{M}^{n,(0)}\nabla_h^+ \big(\Pi^{n,(0)}-\gamma(\nabla_h^- \cdot \mathbf{p}^{n,(0)})\big)\}+\mathcal{O}(\epsilon),\\
=-\mathcal{M}^{n,(0)}\nabla_h^+ \big(\Pi^{n,(0)}-\gamma(\nabla_h^- \cdot \mathbf{p}^{n,(0)})\big) + \mathcal{O}(\Delta x) + \mathcal{O}(\epsilon).
\end{aligned}
\end{equation}
Thanks to well-prepared initial data (see Definition \ref{def:well-prepared_init}), from \eqref{eq:p_updt_tog1} we get
\begin{equation}
\label{eq:q_updt_tog-lim_p_grdu}
    \mathbf{p}_{\mathbf{j}}^{\,n+1,(0)} - (\nabla_h^+ u^{n+1,(0)} )_{\mathbf{j}}= \mathcal{O}(\Delta t\Delta x)+\mathcal{O}(\epsilon),
\end{equation}
for all $n$ and $\mathbf{j}$.

It is thus straightforward to observe that \eqref{eq:h_updt_tog1} together with \eqref{eq:q_updt_tog-lim2} and \eqref{eq:q_updt_tog-lim_p_grdu} is a consistent numerical scheme for \eqref{eq: main} in the asymptotic limit.
\end{proof}

\subsection{Energy stability of the numerical scheme \eqref{eq:semi-imp}-\eqref{eq:expl}}
In this section, we prove that the numerical scheme as defined by \eqref{eq:semi-imp}-\eqref{eq:expl} satisfies a discrete energy identity under a CFL-type condition. We first define the discrete energy at the $n$-th timestep as follows.
\begin{equation}\label{eq:eng_id}
\begin{aligned}
E^n=\sum_{\mathbf{j}}\Delta x^d E^n_{\mathbf{j}}=\sum_{\mathbf{j}}\Delta x^d\bigg(W(u_{\mathbf{j}}^n)+\dfrac{\epsilon{|\psi_{\mathbf{j}}^n|}^2}{2}+\dfrac{\epsilon {|\mathbf{q}_{\mathbf{j}}^n|}^2}{2}+\dfrac{\epsilon {|w_{\mathbf{j}}^n|}^2}{2}+\dfrac{\gamma {|\mathbf{p}_{\mathbf{j}}^n|}^2}{2}\bigg),\\
:=\sum_{\mathbf{j}}\Delta x^d\big(E_{\mathbf{j}}^n(u)+E_{\mathbf{j}}^n(\psi)+E_{\mathbf{j}}^n(\mathbf{q})+E_{\mathbf{j}}^n(w)+E_{\mathbf{j}}^n(\mathbf{p})\big).
\end{aligned}
\end{equation}
Note that $W$ is a convex differentiable function such that $W'(u)=\Pi(u)$ for all $u>0$. Also, $\mathcal{M}(u)>0$ for all $u>0$. In view of the explicit definition of the discrete energy, we now prove the following theorem.

\begin{theorem}[Discrete energy stability]
\label{thm:eng-stable}
Let $\mathbf{U}^{n+1}$ be the numerical approximation generated by the scheme
\eqref{eq:semi-imp}-\eqref{eq:expl} and $\mathcal{M}_{\mathbf{j}}^n>0$ for all $j\in \mathbb{Z}$ and $n\geq 0$. Then, under the CFL condition
\begin{equation}
\label{eq:cfl_stability}
\max\left\{
4\sqrt{d}\sqrt{W''(u_{\mathbf{j}}^{n+\frac12})},
4d\sqrt{d},
d\big(\epsilon+W''(u_{\mathbf{j}}^{n+\frac12})\big),
\gamma,
4\sqrt{d\gamma}
\right\}
\frac{\Delta t}{\Delta x}
\leq 1-\theta,
\end{equation}
the following discrete energy inequality holds
\begin{equation}
    \label{eq:disc_energy-ineq}
    \frac{E^{n+1}-E^{n}}{\Delta t}
    +\frac{\theta}{r}\sum_{\mathbf{j}}\Delta x^d|r\nabla_h^{+}\mathbf{q}^*|_{\mathbf{j}}^2+\frac{\theta}{s}\sum_{\mathbf{j}}\Delta x^d|s\nabla_h^+ w^*|_{\mathbf{j}}^2+
    \sum_{\mathbf{j}}\Delta x^d\,
    \frac{|\mathbf{q}_{\mathbf{j}}^*|^2}{\mathcal{M}_{\mathbf{j}}^n}
    \leq 0.
\end{equation}
In \eqref{eq:cfl_stability}, $u^{n+\half}_{\mathbf{j}}\in\llbracket u_{\mathbf{j}}^n, u_{\mathbf{j}}^{n+1}\rrbracket$ and $0<\theta<1$.
\end{theorem}
To prove Theorem \ref{thm:eng-stable}, we first establish the following local energy identity.
\begin{lemma}[Local energy identity]\label{lemma: local_energy_id}
Any solution of the scheme \eqref{eq:semi-imp}-\eqref{eq:expl} satisfies the identity
\begin{equation}\label{eq: local_energy_id}
\begin{aligned}
\frac{E_{\mathbf{j}}^{n+1}-E_{\mathbf{j}}^{n}}{\Delta t}
+(\nabla_h^{-} \cdot \mathbf{q}^*)_{\mathbf{j}}(\Pi^n_{\mathbf{j}}+\psi^*_{\mathbf{j}})
+(\nabla_h^+(\Pi^n+\psi^*))_{\mathbf{j}}\cdot \mathbf{q}^*_{\mathbf{j}} \\-\gamma(\nabla_h^{-} \cdot \mathbf{p}^*)_{\mathbf{j}}w^*_{\mathbf{j}}-\gamma(\nabla_h^+ w^*)_{\mathbf{j}}\cdot \mathbf{p}^*_{\mathbf{j}}=-\frac{|\mathbf{q}_{j}^{*}|^2}{\mathcal{M}_{\mathbf{j}}^n}+R_{\mathbf{j}}^{n+1},
\end{aligned}
\end{equation}
where $R_{\mathbf{j}}^{n+1}$ is given by
\begin{equation}\label{eq:eng_remainder}
\begin{aligned}
R_{\mathbf{j}}^{n+1}
&=
\frac{(u_{\mathbf{j}}^{n+1}-u_{\mathbf{j}}^n)^2}{2 \Delta t}
W''(u_{\mathbf{j}}^{n+\frac12})
+\epsilon\frac{(\psi_{\mathbf{j}}^{n+1}-\psi^*_{\mathbf{j}})^2}{2\Delta t}
-\epsilon\frac{(\psi^*_{\mathbf{j}}-\psi^n_{\mathbf{j}})^2}{2\Delta t} \\
&+\epsilon\frac{|\mathbf{q}_{\mathbf{j}}^{n+1}-\mathbf{q}_{\mathbf{j}}^{*}|^2}{2\Delta t}
-\epsilon\frac{|\mathbf{q}_{\mathbf{j}}^{*}-\mathbf{q}_{\mathbf{j}}^{n}|^2}{2\Delta t}
+\epsilon\frac{(w_{\mathbf{j}}^{n+1}-w_{\mathbf{j}}^{*})^2}{2\Delta t}
-\epsilon\frac{(w_{\mathbf{j}}^{*}-w_{\mathbf{j}}^{n})^2}{2\Delta t} \\
&+\gamma\frac{|\mathbf{p}_{\mathbf{j}}^{n+1}-\mathbf{p}_{\mathbf{j}}^{*}|^2}{2\Delta t}
+r_1(\Delta_h \Pi^*)_{\mathbf{j}}\Pi^*_{\mathbf{j}}
+\epsilon r_2(\Delta_h \psi^*)_{\mathbf{j}}\psi^*_{\mathbf{j}}
+\gamma r_3(\Delta_h \mathbf{p}^*)_{\mathbf{j}}\cdot \mathbf{p}^*_{\mathbf{j}} \\
&+r(\Delta_{h} \mathbf{q}^*)_{\mathbf{j}}\cdot \mathbf{q}^*_{\mathbf{j}}
+s(\Delta_{h}w^*)_{\mathbf{j}}w^*_{\mathbf{j}} .
\end{aligned}
\end{equation}
\end{lemma}
\begin{proof}
We rewrite the term $\frac{E_{\mathbf{j}}^{n+1}-E_{\mathbf{j}}^{n}}{\Delta t}$ as
\begin{multline}
\label{eq:eng-split-proof}
\frac{E_{\mathbf{j}}^{n+1}-E_{\mathbf{j}}^{n}}{\Delta t}
=\frac{E_{\mathbf{j}}^{n+1}(u)-E_{\mathbf{j}}^{n}(u)}{\Delta t}+\frac{E_{\mathbf{j}}^{n+1}(\psi)-E_{\mathbf{j}}^{n}(\psi)}{\Delta t}\\
+\frac{E_{\mathbf{j}}^{n+1}(\mathbf{q})-E_{\mathbf{j}}^{n}(\mathbf{q})}{\Delta t}+\frac{E_{\mathbf{j}}^{n+1}(w)-E_{\mathbf{j}}^{n}(w)}{\Delta t}+\frac{E_{\mathbf{j}}^{n+1}(\mathbf{p})-E_{\mathbf{j}}^{n}(\mathbf{p})}{\Delta t}.
\end{multline}
We first consider the first term in \eqref{eq:eng-split-proof}. To this end, we apply the Taylor series formulae and obtain
\begin{align*}
\frac{E_{\mathbf{j}}^{n+1}(u)-E_{\mathbf{j}}^{n}(u)}{\Delta t}=\frac{W(u_{\mathbf{j}}^{n+1})-W(u_{\mathbf{j}}^{n})}{\Delta t}
=\frac{u_{\mathbf{j}}^{n+1}-u_{\mathbf{j}}^{n}}{\Delta t}W'(u_{\mathbf{j}}^{n})+\frac{(u_{\mathbf{j}}^{n+1}-u_{\mathbf{j}}^n)^2}{2 \Delta t}W^{''}(u_{\mathbf{j}}^{n+\half}).
\end{align*} 
Using \eqref{eq:h_updt} in the first term, we then obtain
\begin{equation}
\label{eq:h-energy-proof}
\frac{E_{\mathbf{j}}^{n+1}(u)-E_{\mathbf{j}}^{n}(u)}{\Delta t}=-(\nabla_h^{-} \cdot \mathbf{q}^*)_{\mathbf{j}}\Pi^n_{\mathbf{j}}+r_1(\Delta_h \Pi^n)_{\mathbf{j}}\Pi^n_{\mathbf{j}}+\frac{(u_{\mathbf{j}}^{n+1}-u_{\mathbf{j}}^n)^2}{2 \Delta t}W^{''}(u_{\mathbf{j}}^{n+\half}),
\end{equation}
where we have used the fact that $u^*=u^n$ and $\Pi^*=\Pi^n$.

Now for the contribution of the term involving $\psi$ in \eqref{eq:eng-split-proof}, we compute
\begin{align}\label{psi_energy}
\frac{E_{\mathbf{j}}^{n+1}(\psi)-E_{\mathbf{j}}^{n}(\psi)}{\Delta t}&=\epsilon\frac{|\psi_{\mathbf{j}}^{n+1}|^2-|\psi^n_{\mathbf{j}}|^2}{2\Delta t}=\epsilon\frac{|\psi_{\mathbf{j}}^{n+1}|^2-|\psi^*_{\mathbf{j}}|^2}{2\Delta t}+\epsilon\frac{|\psi^*_{\mathbf{j}}|^2-|\psi^n_{\mathbf{j}}|^2}{2\Delta t},\nonumber\\
&=\epsilon\frac{(\psi_{\mathbf{j}}^{n+1}-\psi^*_{\mathbf{j}})\psi^*_{\mathbf{j}}}{\Delta t}+\epsilon\frac{(\psi_{\mathbf{j}}^{n+1}-\psi^*_{\mathbf{j}})^2}{2\Delta t}+\epsilon\frac{(\psi^*_{\mathbf{j}}-\psi^n_{\mathbf{j}})\psi^*_{\mathbf{j}}}{\Delta t}-\epsilon\frac{(\psi^*_{\mathbf{j}}-\psi^n_{\mathbf{j}})^2}{2\Delta t},
\end{align}
where for any $\mathbf{a}^d,\mathbf{a}^n\in \mathbb{R}^n, ~n\geq 1$, we have used the following algebraic identities
\begin{equation}\label{eq: square_identity_1}
(\mathbf{a}^d-\mathbf{a}^n)\cdot \mathbf{a}^d=\frac12\Big(|\mathbf{a}^d|^2-|\mathbf{a}^n|^2+|\mathbf{a}^d-\mathbf{a}^n|^2\Big),
\end{equation}
and 
\begin{equation}\label{eq: square_identity_2}
(\mathbf{a}^d-\mathbf{a}^n)\cdot \mathbf{a}^n=\frac12\Big(|\mathbf{a}^d|^2-|\mathbf{a}^n|^2-|\mathbf{a}^d-\mathbf{a}^n|^2\Big).
\end{equation}
Thus, using \eqref{eq:psi_updt_star} and \eqref{eq:psi_updt} in \eqref{psi_energy}, we have
\begin{equation}
\begin{aligned}
\label{eq:psi-energy-proof}
\frac{E_{\mathbf{j}}^{n+1}(\psi)-E_{\mathbf{j}}^{n}(\psi)}{\Delta t}&=-\epsilon(\nabla_h^{-} \cdot \mathbf{q}^*)_{\mathbf{j}}\psi^*_{\mathbf{j}}+\epsilon r_2(\Delta_h \psi^*)_{\mathbf{j}}\psi^*_{\mathbf{j}}-w^*_{\mathbf{j}}\psi^*_{\mathbf{j}}-(1-\epsilon)(\nabla_h^{-} \cdot \mathbf{q}^*)_{\mathbf{j}}\psi^*_{\mathbf{j}}\\
&\hspace{ 4 cm}+\epsilon\frac{(\psi_{\mathbf{j}}^{n+1}-\psi^*_{\mathbf{j}})^2}{2\Delta t}-\epsilon\frac{(\psi^*_{\mathbf{j}}-\psi^n_{\mathbf{j}})^2}{2\Delta t},\\
&=-(\nabla_h^{-} \cdot \mathbf{q}^*)_{\mathbf{j}}\psi^*_{\mathbf{j}}+\epsilon r_2(\Delta_h \psi^*)_{\mathbf{j}}\psi^*_{\mathbf{j}}-w^*_{\mathbf{j}}\psi^*_{\mathbf{j}}+\epsilon\frac{(\psi_{\mathbf{j}}^{n+1}-\psi^*_{\mathbf{j}})^2}{2\Delta t}-\epsilon\frac{(\psi^*_{\mathbf{j}}-\psi^n_{\mathbf{j}})^2}{2\Delta t}.
\end{aligned}
\end{equation}
Next, we evaluate the third term in \eqref{eq:eng-split-proof} using the identities \eqref{eq: square_identity_1} and \eqref{eq: square_identity_2} as
\begin{align*}
\frac{E_{\mathbf{j}}^{n+1}(\mathbf{q})-E_{\mathbf{j}}^{n}(\mathbf{q})}{\Delta t}&=\epsilon\frac{|\mathbf{q}_{\mathbf{j}}^{n+1}|^2-|\mathbf{q}_{\mathbf{j}}^{n}|^2}{2\Delta t}=\epsilon\frac{|\mathbf{q}_{\mathbf{j}}^{n+1}|^2-|\mathbf{q}_{\mathbf{j}}^{*}|^2}{2\Delta t}+\epsilon\frac{|\mathbf{q}_{\mathbf{j}}^{*}|^2-|\mathbf{q}_{\mathbf{j}}^{n}|^2}{2\Delta t},\\
&=\epsilon\frac{(\mathbf{q}_{\mathbf{j}}^{n+1}-\mathbf{q}_{\mathbf{j}}^{*})\cdot\mathbf{q}_{\mathbf{j}}^{*}}{\Delta t}+\epsilon\frac{|\mathbf{q}_{\mathbf{j}}^{n+1}-\mathbf{q}_{\mathbf{j}}^{*}|^2}{2\Delta t}+\epsilon\frac{(\mathbf{q}_{\mathbf{j}}^{*}-\mathbf{q}_{\mathbf{j}}^{n})\cdot \mathbf{q}_{\mathbf{j}}^{*}}{\Delta t}-\epsilon\frac{|\mathbf{q}_{\mathbf{j}}^{*}-\mathbf{q}_{\mathbf{j}}^{n}|^2}{2\Delta t}.
\end{align*}
Using \eqref{eq:q_updt_star} and \eqref{eq:q_updt}, we obtain 
\begin{equation}
\begin{aligned}
\label{eq:q-energy-proof}
\frac{E_{\mathbf{j}}^{n+1}(\mathbf{q})-E_{\mathbf{j}}^{n}(\mathbf{q})}{\Delta t}=-(\nabla_h^+(\Pi^n+\psi^*))_{\mathbf{j}}\cdot \mathbf{q}^*_{\mathbf{j}}+ r(\Delta_{h} \mathbf{q}^*)_{\mathbf{j}}\cdot \mathbf{q}^*_{\mathbf{j}}-\frac{|\mathbf{q}_{j}^{*}|^2}{\mathcal{M}_{\mathbf{j}}^n}+\epsilon\frac{|\mathbf{q}_{\mathbf{j}}^{n+1}-\mathbf{q}_{\mathbf{j}}^{*}|^2}{2\Delta t}-\epsilon\frac{|\mathbf{q}_{\mathbf{j}}^{*}-\mathbf{q}_{\mathbf{j}}^{n}|^2}{2\Delta t}.
\end{aligned}
\end{equation}
In order to evaluate the fourth term in \eqref{eq:eng-split-proof}, we proceed along the same line as the previous terms and obtain
\begin{align*}
\frac{E_{\mathbf{j}}^{n+1}(w)-E_{\mathbf{j}}^{n}(w)}{\Delta t}&=\epsilon\frac{|w_{\mathbf{j}}^{n+1}|^2-|w_{\mathbf{j}}^{n}|^2}{2\Delta t}=\epsilon\frac{|w_{\mathbf{j}}^{n+1}|^2-|w_{\mathbf{j}}^{*}|^2}{2\Delta t}+\epsilon\frac{|w_{\mathbf{j}}^{*}|^2-|w_{\mathbf{j}}^{n}|^2}{2\Delta t},\\
&=\epsilon\frac{(w_{\mathbf{j}}^{n+1}-w_{\mathbf{j}}^{*})w_{\mathbf{j}}^{*}}{\Delta t}+\epsilon\frac{(w_{\mathbf{j}}^{n+1}-w_{\mathbf{j}}^{*})^2}{2\Delta t}+\epsilon\frac{(w_{\mathbf{j}}^{*}-w_{\mathbf{j}}^{n})w_{\mathbf{j}}^*}{\Delta t}-\epsilon\frac{(w_{\mathbf{j}}^{*}-w_{\mathbf{j}}^{n})^2}{2\Delta t}.
\end{align*}
Using \eqref{eq:w_updt_star} and \eqref{eq:w_updt}, we obtain
\begin{equation}
\begin{aligned}
\label{eq:w-energy-proof}
\frac{E_{\mathbf{j}}^{n+1}(w)-E_{\mathbf{j}}^{n}(w)}{\Delta t}=\gamma(\nabla_h^{-} \cdot \mathbf{p}^*)_{\mathbf{j}}w^*_{\mathbf{j}}+ s(\Delta_{h}w^*)_{\mathbf{j}}w^*_{\mathbf{j}}+\psi^*_{\mathbf{j}}w_{\mathbf{j}}^*+\epsilon\frac{(w_{\mathbf{j}}^{n+1}-w_{\mathbf{j}}^{*})^2}{2\Delta t}-\epsilon\frac{(w_{\mathbf{j}}^{*}-w_{\mathbf{j}}^{n})^2}{2\Delta t}.
\end{aligned}
\end{equation}
Finally, we evaluate the fifth term in \eqref{eq:eng-split-proof} as
\begin{align*}
\frac{E_{\mathbf{j}}^{n+1}(\mathbf{p})-E_{\mathbf{j}}^{n}(\mathbf{p})}{\Delta t}&=\gamma\frac{|\mathbf{p}_{\mathbf{j}}^{n+1}|^2-|\mathbf{p}_{\mathbf{j}}^{n}|^2}{2\Delta t}=\gamma\frac{(\mathbf{p}_{\mathbf{j}}^{n+1}-\mathbf{p}_{\mathbf{j}}^{*})\cdot \mathbf{p}_{\mathbf{j}}^{*}}{\Delta t}+\gamma\frac{|\mathbf{p}_{\mathbf{j}}^{n+1}-\mathbf{p}_{\mathbf{j}}^{*}|^2}{2\Delta t}.
\end{align*}
Then using \eqref{eq:p_updt_star} and \eqref{eq:p_updt}, we further obtain
\begin{equation}
\begin{aligned}
\label{eq:p-energy-proof}
\frac{E_{\mathbf{j}}^{n+1}(\mathbf{p})-E_{\mathbf{j}}^{n}(\mathbf{p})}{\Delta t}=\gamma(\nabla_h^+ w^*)_{\mathbf{j}}\cdot \mathbf{p}^*_{\mathbf{j}}+ \gamma r_3(\Delta_h \mathbf{p}^*)_{\mathbf{j}}\cdot \mathbf{p}^*_{\mathbf{j}}+\gamma\frac{|\mathbf{p}_{\mathbf{j}}^{n+1}-\mathbf{p}_{\mathbf{j}}^{*}|^2}{2\Delta t}.
\end{aligned}
\end{equation}
Finally, adding \eqref{eq:h-energy-proof}, \eqref{eq:psi-energy-proof}, \eqref{eq:q-energy-proof}, \eqref{eq:w-energy-proof}, and \eqref{eq:p-energy-proof}, the terms $-w_{\mathbf{j}}^*\psi^*_{\mathbf{j}}$ and $\psi^*_{\mathbf{j}}w_{\mathbf{j}}^*$ cancel. Rearranging the remaining
terms gives \eqref{eq: local_energy_id}. This proves the lemma.
\end{proof}
With this Lemma in place, we now proceed to prove the Theorem  \ref{thm:eng-stable} as follows.
\begin{proof}[Proof of Theorem  \ref{thm:eng-stable}]
Taking sum of all $\mathbf{j}$ in \eqref{eq:disc_energy-ineq} yields
\begin{equation}
\label{eq:sum_eng-id}
\sum_{\mathbf{j}}\left(\frac{E_{\mathbf{j}}^{n+1}-E_{\mathbf{j}}^{n}}{\Delta t}\right)=-\sum_{\mathbf{j}}\frac{|\mathbf{q}_{j}^{*}|^2}{\mathcal{M}_{\mathbf{j}}^n}+\sum_{\mathbf{j}}R_{\mathbf{j}}^{n+1}.
\end{equation}
Using the identity $\sum_{\mathbf{j}}(\Delta_h f)_{\mathbf{j}}f_{\mathbf{j}}=-\sum_{j}\lVert \nabla_h^\pm f\rVert^2_{\mathbf{j}}$ in view of Lemma \ref{lemma:discrete_ibp}, we bound $\sum_{\mathbf{j}}R_{\mathbf{j}}^{n+1}$ as
\[
\begin{aligned}
\sum_{\mathbf{j}}R_{\mathbf{j}}^{n+1}
&\leq
\sum_{\mathbf{j}} \frac{(u_{\mathbf{j}}^{n+1}-u_{\mathbf{j}}^n)^2}{2\Delta t}
W''(u_{\mathbf{j}}^{n+\frac12})
+\epsilon\sum_{\mathbf{j}}\frac{(\psi_{\mathbf{j}}^{n+1}-\psi^*_{\mathbf{j}})^2}{2\Delta t} \\
&\quad
+\epsilon\sum_{\mathbf{j}}\frac{|\mathbf q_{\mathbf{j}}^{n+1}-\mathbf q_{\mathbf{j}}^*|^2}{2\Delta t}
+\epsilon\sum_{\mathbf{j}}\frac{(w_{\mathbf{j}}^{n+1}-w_{\mathbf{j}}^*)^2}{2\Delta t}
+\gamma\sum_{\mathbf{j}}\frac{|\mathbf p_{\mathbf{j}}^{n+1}-\mathbf p_{\mathbf{j}}^*|^2}{2\Delta t} \\
&\quad
-r_1\sum_{\mathbf{j}}|\nabla_h^+\Pi^*|_{\mathbf{j}}^2
-\epsilon r_2\sum_{\mathbf{j}}|\nabla_h^+\psi^*|_{\mathbf{j}}^2
 -\gamma r_3\sum_{\mathbf{j}}|\nabla_h^+\mathbf p^*|_{\mathbf{j}}^2 
-r\sum_\mathbf{j}|\nabla_h^+\mathbf q^*|_{\mathbf{j}}^2
-s\sum_{\mathbf{j}}|\nabla_h^+w^*|_{\mathbf{j}}^2 .
\end{aligned}
\]
Using \eqref{eq:semi-imp}-\eqref{eq:expl} and the classical inequality $(a+b)^2\leq 2(a^2+b^2)$, wherever needed we then obtain
\begin{align}\label{eq: main_remainders}
\sum_{\mathbf{j}}R_{\mathbf{j}}^{n+1}&\leq \Delta t \sum_{\mathbf{j}} r_1^2(\Delta_h \Pi^*)_{\mathbf{j}}^2W^{''}(u_{\mathbf{j}}^{n+\half})+\Delta t\sum_{\mathbf{j}}|\nabla_h^{-} \cdot \mathbf{q}^*|_{\mathbf{j}}^2W^{''}(u_{\mathbf{j}}^{n+\half})-r_1\sum_{\mathbf{j}}|\nabla_h^+ \Pi^*|^2_{\mathbf{j}}\nonumber\\
&+\Delta t \epsilon\sum_{\mathbf{j}}|\nabla_h^+ \Pi^*|_{\mathbf{j}}^2+\Delta t \epsilon\sum_{\mathbf{j}}|\nabla_h^+ \psi^*|_{\mathbf{j}}^2 -r\sum_{\mathbf{j}}|\nabla_h^{+}\mathbf{q}^*|_{\mathbf{j}}^2\nonumber\\
&+\Delta t\epsilon\sum_{\mathbf{j}}r_2^2(\Delta_h \psi^*)_\mathbf{j}^2+\Delta t \epsilon\sum_{\mathbf{j}}|\nabla_h^{-} \cdot \mathbf{q}^*|_{\mathbf{j}}^2-r_2\epsilon\sum_{\mathbf{j}}|\nabla_h^+ \psi^*|_{\mathbf{j}}^2\nonumber\\
&+\Delta t \epsilon\sum_{\mathbf{j}}\gamma^2|\nabla_h^- \cdot\mathbf{p}^*|_{\mathbf{j}}^2-s\sum_{\mathbf{j}}|\nabla_h^+ w^*|_{\mathbf{j}}^2\nonumber\\
&+\sum_{\mathbf{j}}\Delta t \gamma r_3^2|\Delta_h \mathbf{p}^*|_{\mathbf{j}}^2+\Delta t \gamma\sum_{\mathbf{j}}|\nabla_h^+ w^*|_{\mathbf{j}}^2 -\gamma r_3\sum_{\mathbf{j}}|\nabla_h^{+} \mathbf{p}^*|_{\mathbf{j}}^2\nonumber\\
&:=R^{n+1}(u)+R^{n+1}(\psi)+R^{n+1}(\mathbf q)+R^{n+1}(w)+R^{n+1}(\mathbf p),
\end{align}
where
\begin{align}
R^{n+1}(u)
&:=
\Delta t r_1^2\sum_{\mathbf{j}}(\Delta_h \Pi^*)_{\mathbf{j}}^2W''(u_{\mathbf{j}}^{n+\frac12})
-r_1\sum_{\mathbf{j}}|\nabla_h^+\Pi^*|_{\mathbf{j}}^2
+\Delta t\epsilon\sum_{\mathbf{j}}|\nabla_h^+\Pi^*|_{\mathbf{j}}^2,
\label{eq:R-h}
\\
R^{n+1}(\psi)
&:=
\Delta t\epsilon r_2^2\sum_{\mathbf{j}}(\Delta_h\psi^*)_{\mathbf{j}}^2
-\epsilon r_2\sum_{\mathbf{j}}|\nabla_h^+\psi^*|_{\mathbf{j}}^2
+\Delta t\epsilon\sum_{\mathbf{j}}|\nabla_h^+\psi^*|_{\mathbf{j}}^2,
\label{eq:R-psi}
\\
R^{n+1}(\mathbf q)
&:=
\Delta t\sum_{\mathbf{j}}|\nabla_h^{-}\cdot\mathbf q^*|_{\mathbf{j}}^2W''(u_{\mathbf{j}}^{n+\frac12})
+\Delta t\epsilon\sum_{\mathbf{j}}|\nabla_h^{-}\cdot\mathbf q^*|_{\mathbf{j}}^2
 -r\sum_{\mathbf{j}}|\nabla_h^{+}\mathbf q^*|_{\mathbf{j}}^2,
\label{eq:R-q}
\\
R^{n+1}(w)
&:=
\Delta t\gamma\sum_{\mathbf{j}}|\nabla_h^+w^*|_{\mathbf{j}}^2
-s\sum_{\mathbf{j}}|\nabla_h^+w^*|_{\mathbf{j}}^2,
\label{eq:R-w}
\\
R^{n+1}(\mathbf p)
&:=
\Delta t\gamma r_3^2\sum_{\mathbf{j}}|\Delta_h\mathbf p^*|_{\mathbf{j}}^2
+\Delta t\epsilon\gamma^2\sum_{\mathbf{j}}|\nabla_h^{-}\cdot\mathbf p^*|_{\mathbf{j}}^2
 -\gamma r_3\sum_{\mathbf{j}}|\nabla_h^{+}\mathbf p^*|_{\mathbf{j}}^2 .
\label{eq:R-p}
\end{align}
In what follows, we prove that each of these remainder terms is non-positive under the CFL condition \eqref{eq:cfl_stability}. To this end, using Lemma \ref{lemma:discrete_inverse} in \eqref{eq:R-h}, we directly compute
\begin{equation}
\begin{aligned}
R^{n+1}(u)&=\Delta t \sum_{\mathbf{j}} r_1^2(\Delta_h \Pi^*)_{\mathbf{j}}^2W^{''}(u_{\mathbf{j}}^{n+\half})-r_1\sum_{\mathbf{j}}|\nabla_h^+ \Pi^*|^2_{\mathbf{j}}+\Delta t \epsilon\sum_{\mathbf{j}}|\nabla_h^+ \Pi^*|_{\mathbf{j}}^2,\\
&\leq \frac{4d\Delta tr_1^2}{\Delta x^2}\sum_{\mathbf{j}}|\nabla_h^+ \Pi^*|_{\mathbf{j}}^2W^{''}(u_{\mathbf{j}}^{n+\half})-r_1\sum_{\mathbf{j}}|\nabla_h^+ \Pi^*|^2_{\mathbf{j}}+\Delta t \epsilon\sum_{\mathbf{j}}|\nabla_h^+ \Pi^*|_{\mathbf{j}}^2.
\end{aligned}
\end{equation}
Taking $r_1=2\Delta t\epsilon$ and as $\epsilon\leq 1$, we have the following estimate in view of the CFL condition \eqref{eq:cfl_stability}
\begin{equation}
\begin{aligned}
\sum_{\mathbf{j}}R_{\mathbf{j}}^{n+1}(u)&\leq \epsilon\Delta t\frac{16d \Delta t^2}{\Delta x^2} \sum_{\mathbf{j}} |\nabla_h^+ \Pi^*|_{\mathbf{j}}^2W^{''}(u_{\mathbf{j}}^{n+\half})-\epsilon\Delta t \sum_{\mathbf{j}}|\nabla_h^+ \Pi^*|_{\mathbf{j}}^2,\\
&\leq \Delta t\epsilon\sum_{\mathbf{j}}\left(\frac{16 d\Delta t^2}{\Delta x^2}W^{''}(u_{\mathbf{j}}^{n+\half})-1\right)|\nabla_h^+ \Pi^*|_{\mathbf{j}}^2\leq 0.
\end{aligned}
\end{equation}
Similarly, we compute
\begin{equation}
\begin{aligned}
R^{n+1}(\psi)&=\Delta t\epsilon \sum_{\mathbf{j}} r_2^2(\Delta_h \psi^*)_{\mathbf{j}}^2-r_2\epsilon\sum_{\mathbf{j}}(\nabla_h^+ \psi^*)^2_{\mathbf{j}}+\Delta t \epsilon\sum_{\mathbf{j}}(\nabla_h^+ \psi^*)_{\mathbf{j}}^2,\\
&\leq \frac{4\epsilon d \Delta t r_2^2}{\Delta x^2}\sum_{\mathbf{j}}|\nabla_h^+ \psi^*|_{\mathbf{j}}^2-\epsilon r_2\sum_{\mathbf{j}}|\nabla_h^+ \psi^*|^2_{\mathbf{j}}+\Delta t \epsilon\sum_{\mathbf{j}}|\nabla_h^+ \psi^*|_{\mathbf{j}}^2.
\end{aligned}
\end{equation}
Taking $r_2=2\Delta t$, $\epsilon\leq 1$ and using using the CFL condition \eqref{eq:cfl_stability}, we obtain the estimate
\begin{equation}
R^{n+1}(\psi)\leq\Delta t\epsilon\sum_{\mathbf{j}}\left(\frac{16d\Delta t^2}{\Delta x^2}-1\right)|\nabla_h^+ \psi^*|_{\mathbf{j}}^2\leq 0.
\end{equation}
Using \eqref{eq:div_less_grad}, we compute the remainder term in $\mathbf{q}$ as
\begin{equation}
\label{eq:remaind_q-star}
\begin{aligned}
R^{n+1}(\mathbf{q})&=\Delta t \sum_{\mathbf{j}}|\nabla_h^{-} \cdot \mathbf{q}^*|_{\mathbf{j}}^2W^{''}(u_{\mathbf{j}}^{n+\half})-r\sum_{\mathbf{j}}(\nabla_h^+ \mathbf{q}^*)^2_{\mathbf{j}}+\Delta t \epsilon\sum_{\mathbf{j}}|\nabla_h^{-} \cdot \mathbf{q}^*|_{\mathbf{j}}^2,\\
&\leq\sum_{\mathbf{j}}\left(d\big(\epsilon+W^{''}(u_{\mathbf{j}}^{n+\half})\big)\frac{\Delta t}{r}-1\right)r|\nabla_h^{+} \mathbf{q}^*|_{\mathbf{j}}^2.
\end{aligned}
\end{equation}
For $r=\Delta x$ and using the CFL condition \eqref{eq:cfl_stability}, it follows that $R^{n+1}(\mathbf{q})\leq 0$.

Next, we compute
\begin{equation}
\label{eq:remaind_w-star}
R^{n+1}(w)=\left(\gamma\frac{\Delta t}{s}-1\right)s\sum_{\mathbf{j}}|\nabla_h^+ w^*|_{\mathbf{j}}^2.
\end{equation}
Choosing $s=\Delta x$ and using CFL condition \eqref{eq:cfl_stability}, we have $R^{n+1}(w)\leq 0$.

Finally, using \eqref{eq:div_less_grad} we have
\begin{equation}
\begin{aligned}
R^{n+1}(\mathbf p)&= r_3^2\gamma\Delta t \sum_{\mathbf{j}}|\Delta_h \mathbf{p}^*|_{\mathbf{j}}^2 -r_3\gamma\sum_{\mathbf{j}}|\nabla_h^+ \mathbf{p}^*|_{\mathbf{j}}^2+\Delta t \epsilon \gamma^2 \sum_{\mathbf{j}}|\nabla_h^- \cdot\mathbf{p}^*|_{\mathbf{j}}^2,\\
&\leq \gamma\frac{4d\Delta t}{\Delta x^2}r_3^2\sum_{\mathbf{j}}|\nabla_h^+\mathbf{p}^*|_{\mathbf{j}}^2 -\gamma r_3\sum_{\mathbf{j}}|\nabla_h^+\mathbf{p}^*|_{\mathbf{j}}^2+\Delta t\gamma^2\epsilon\sum_{\mathbf{j}}|\nabla_h^- \cdot\mathbf{p}^*|_{\mathbf{j}}^2,\\
&\leq \gamma\frac{4d\Delta t}{\Delta x^2}r_3^2\sum_{\mathbf{j}}|\nabla_h^+\mathbf{p}^*|_{\mathbf{j}}^2-\gamma r_3\sum_{\mathbf{j}}|\nabla_h^+\mathbf{p}^*|_{\mathbf{j}}^2+\Delta t\gamma^2\epsilon d\sum_{\mathbf{j}}|\nabla_h^+\mathbf{p}^*|_{\mathbf{j}}^2.
\end{aligned}
\end{equation}
We use $r_3=2\epsilon d\gamma\Delta t$, $\epsilon\leq 1$ and apply the CFL condition \eqref{eq:cfl_stability} to obtain the estimate
\begin{equation}
 R^{n+1}(\mathbf{p})\leq\gamma^2\epsilon\Delta t\sum_{\mathbf{j}}\left(\frac{16d^3\gamma\Delta t^2}{\Delta x^2}-1\right)|\nabla_h^+\mathbf{p}^*|_{\mathbf{j}}^2\leq 0.
\end{equation}
Thus, from \eqref{eq: main_remainders} we obtain $\sum_{\mathbf{j}}R_{\mathbf{j}}^{n+1}\leq 0$. Injecting this into \eqref{eq:sum_eng-id} and multiplying by $\Delta x^d$, we prove Theorem~\ref{thm:eng-stable}.
\end{proof}
It is important to point out that the proof of Lemma \ref{lemma: num_existence} and Theorem  \ref{thm:eng-stable} relies on the positivity of the discrete mobility $\mathcal{M}_{\mathbf{j}}^n$. For constant positive mobilities, this is trivial. However, in general, the mobility under consideration may degenerate or may fail to remain positive if $u_{\mathbf{j}}^n$ loses positivity. Therefore, it is important to ensure that the numerical film height $u_{\mathbf{j}}^n$ remains positive for all $n\geq 0$. 

Note that no additional time-step restriction is required at the predictor stage for the positivity of $u^n$. Indeed, the variable $u$ does not evolve at this stage, and its predicted value satisfies $u^*=u^n$. Therefore, whenever $u^{n}>0$, we have $u^{*}>0$. However, in the corrector stage, a suitable time-step restriction is required to ensure that the updated explicit value $u^{n+1}$ remains positive. Thus, the positivity condition only imposes a time-step restriction at the corrector stage and ensures that whenever $u^{n}>0$, then we have $u^{n+1}>0$.

We ensure positivity using a time-step restriction in the following Lemma. 
\begin{lemma}[Positivity of the numerical solution $u_{\mathbf{j}}^n$]
\label{lemma: positivity}
    Assume that the discrete solution satisfies $u^n_\mathbf{j}>0$ for all $\mathbf{j}$ and some $n\geq 0$. Then the corrector step as defined in \eqref{eq:expl} admits a positive solution at the next time step i.e.\ $u_\mathbf{j}^{n+1}>0$ for all $\mathbf{j}$ under the following timestep restriction
\begin{equation}
  \label{eq:suff_tstep_positivity}
\frac{\Delta t}{\Delta x}
\max_\alpha\left(
\left| \mathbf{q}_{\mathbf{j}}^{*,\alpha} \right|
+
\left| \mathbf{q}_{\mathbf{j}-\mathbf{e}_\alpha}^{*,\alpha} \right|
+
\sqrt{2\left|
\Pi_{\mathbf{j}+\mathbf{e}_\alpha}^*
-2\Pi_{\mathbf{j}}^*
+\Pi_{\mathbf{j}-\mathbf{e}_\alpha}^*
\right|}
\right)
\leq
\frac{1}{d}\min\{1, u^*_{\mathbf j}\},
\end{equation}
\end{lemma}
\begin{proof}
  Note that the inequality \eqref{eq:suff_tstep_positivity} implies, for each $\alpha \in \{1,2,\dots,d\}$,
    \begin{equation}
    \label{eq:timestep_height}
    \frac{\Delta t}{\Delta x}
\left(
\left| \mathbf{q}_{\mathbf{j}}^{*,\alpha} \right|
+
\left| \mathbf{q}_{\mathbf{j}-\mathbf{e}_\alpha}^{*,\alpha} \right|
+
\sqrt{2 \left|
\Pi_{\mathbf{j}+\mathbf{e}_\alpha}^*
-2\Pi_{\mathbf{j}}^*
+\Pi_{\mathbf{j}-\mathbf{e}_\alpha}^*
\right|}
\right)
\leq
\frac{1}{d}\min\{1, u^*_{\mathbf j}\}, 
    \end{equation}
    and further 
\begin{equation}
\label{eq:timestep_posit_temp}
\frac{\Delta t}{\Delta x}
\left(
\left| \mathbf{q}_{\mathbf{j}}^{*,\alpha} - \mathbf{q}_{\mathbf{j}-\mathbf{e}_\alpha}^{*,\alpha} \right|
+
\sqrt{2 \left|
\Pi_{\mathbf{j}+\mathbf{e}_\alpha}^*
-2\Pi_{\mathbf{j}}^*
+\Pi_{\mathbf{j}-\mathbf{e}_\alpha}^*
\right|}
\right)
\leq
\frac{1}{d}\min\{1, u^*_{\mathbf j}\}. 
\end{equation}
Moreover, since $\epsilon\leq1$, from \eqref{eq:timestep_posit_temp} we have
\begin{equation}
\label{eq:timestep_posit_temp1}
\frac{\Delta t}{\Delta x}
\left(
\left| \mathbf{q}_{\mathbf{j}}^{*,\alpha} - \mathbf{q}_{\mathbf{j}-\mathbf{e}_\alpha}^{*,\alpha} \right|
+
\sqrt{2 \epsilon\left|
\Pi_{\mathbf{j}+\mathbf{e}_\alpha}^*
-2\Pi_{\mathbf{j}}^*
+\Pi_{\mathbf{j}-\mathbf{e}_\alpha}^*
\right|}
\right)
\leq
\frac{1}{d}\min\{1, u^*_{\mathbf j}\}. 
\end{equation}
Now, as the right side of \eqref{eq:timestep_posit_temp1} is less than or equal to
$1$, we get
\begin{equation}
\label{eq:time-step_proof-im}
    \frac{\Delta t}{\Delta x}
\left| \mathbf{q}_{\mathbf{j}}^{*,\alpha} - \mathbf{q}_{\mathbf{j}-\mathbf{e}_\alpha}^{*,\alpha} \right|
+
\frac{\Delta t^2}{\Delta x^2}2\epsilon \left|
\Pi_{\mathbf{j}+\mathbf{e}_\alpha}^*
-2\Pi_{\mathbf{j}}^*
+\Pi_{\mathbf{j}-\mathbf{e}_\alpha}^*
\right|
\leq
\frac{1}{d}u^*_{\mathbf j}. 
\end{equation}
As $a\leq |a|$ and $-a\leq |a|$ for any real number $a$, \eqref{eq:time-step_proof-im} leads to 
\begin{equation}
\label{eq:time-step_proof-last}
    \frac{\Delta t}{\Delta x}
\left( \mathbf{q}_{\mathbf{j}}^{*,\alpha} - \mathbf{q}_{\mathbf{j}-\mathbf{e}_\alpha}^{*,\alpha} \right)
-
\frac{\Delta t^2}{\Delta x^2}2\epsilon \left(
\Pi_{\mathbf{j}+\mathbf{e}_\alpha}^*
-2\Pi_{\mathbf{j}}^*
+\Pi_{\mathbf{j}-\mathbf{e}_\alpha}^*
\right)
\leq
\frac{1}{d}u^*_{\mathbf j}. 
\end{equation}

Finally, we use \eqref{eq:h_updt_star} and \eqref{eq:time-step_proof-last} to get
\begin{align}
u_{\mathbf{j}}^{n+1}&=
u^*_{\mathbf j}-\sum_{\alpha=1}^d\frac{\Delta t}{\Delta x}
\left( \mathbf{q}_{\mathbf{j}}^{*,\alpha} - \mathbf{q}_{\mathbf{j}-\mathbf{e}_\alpha}^{*,\alpha} \right)
+
\sum_{\alpha=1}^d\frac{\Delta t^2}{\Delta x^2}2\epsilon \left(
\Pi_{\mathbf{j}+\mathbf{e}_\alpha}^*
-2\Pi_{\mathbf{j}}^*
+\Pi_{\mathbf{j}-\mathbf{e}_\alpha}^*
\right)\geq 0. 
\end{align}
\end{proof}
\begin{remark}
    Note that the timestep condition \eqref{eq:suff_tstep_positivity} is not required for energy stability. Furthermore, because this condition is implicit during the predictor step \eqref{eq:semi-imp}, it is only used while computing the corrector step \eqref{eq:expl}. However, the timestep condition \eqref{eq:cfl_stability} applies to both steps.
\end{remark}
\begin{remark}
It is worth noting the different mechanisms by which positivity is established for the numerical approximation of the relaxation system \eqref{hyperbolic_system} and for the limiting equation \eqref{eq: main}. For the limit equation \eqref{eq: main}, existing positivity-(or non-negativity-) preserving numerical schemes, such as those developed by Zhornitskaya \& Bertozzi \cite{zhornitskaya1999positivity}, Gr\"un \& Rumpf \cite{grun2000nonnegativity}, and Kim et al. \cite{kim2024positivity}, typically rely on a discrete entropy inequality. In contrast, for the numerical approximation defined by \eqref{eq:semi-imp}-\eqref{eq:expl}, positivity of the numerical solution does not require a discrete entropy inequality. Instead, it follows directly from the structure of the scheme, provided an additional CFL condition is satisfied.
\end{remark}

\subsection{Solvability of the numerical scheme}
\label{sec:solvability}
This section describes the implementation of the scheme \eqref{eq:semi-imp}-\eqref{eq:expl}. The implicit predictor step \eqref{eq:semi-imp} has the unknowns $\mathbf{q}^*$, $\psi^*$, $w^*$, while $u^*$ and $\mathbf{p}^*$ can be explicitly solved. However, using \eqref{eq:q_updt_star}, \eqref{eq:psi_updt_star} and \eqref{eq:w_updt_star}, we first eliminate $w^*$ thereby reducing the predictor step to a linear system for $(\psi^*,\mathbf{q}^*)$. To clarify more, we first rewrite the $w^*$ predictor update \eqref{eq:w_updt_star} as

\[
L_w w^*
=
w^n + \frac{\Delta t}{\epsilon}\psi^*
+ \frac{\Delta t}{\epsilon}\gamma(1-\epsilon)(\nabla_h^- \cdot \mathbf{p}^{*}),
\]
where $L_w$ is the discrete operator defined as
\[
L_w := I - \frac{\Delta t}{\epsilon}s\,\Delta_h,
\]
which is a negative definite matrix under periodic boundary conditions. This implies that the matrix associated with the operator $L_w$ is a positive definite matrix and thus is invertible. This allows us to find $w^*$ as 
\begin{equation}
\label{eq:wstar_recovery}
w^*
=
L_w^{-1}
\left(
w^n + \frac{\Delta t}{\epsilon}\psi^*
+ \frac{\Delta t}{\epsilon}\gamma(1-\epsilon)(\nabla_h^- \cdot \mathbf{p}^{*})
\right).
\end{equation}
Now, we use the value of $w^*$ in \eqref{eq:psi_updt_star} to obtain 
\[
\psi^*
=
\psi^n
-
\frac{\Delta t}{\epsilon}(1-\epsilon)\nabla_h^- \cdot \mathbf{q}^*
-
\frac{\Delta t}{\epsilon}
L_w^{-1}
\left(
w^n + \frac{\Delta t}{\epsilon}\psi^*
+ \frac{\Delta t}{\epsilon}\gamma(1-\epsilon)\nabla_h^- \cdot \mathbf{p}^n
\right).
\]
After rearranging terms, we then obtain
\[
\left(
I + \frac{\Delta t^2}{\epsilon^2}L_w^{-1}
\right)\psi^*
+
\frac{\Delta t}{\epsilon}(1-\epsilon)\nabla_h^- \cdot \mathbf{q}^*
=
\psi^n
-
\frac{\Delta t}{\epsilon}
L_w^{-1}
\left(
w^n + \frac{\Delta t}{\epsilon}\gamma(1-\epsilon)\nabla_h^- \cdot \mathbf{p}^n
\right).
\]
Since $L_w$ is a positive definite operator so as $L_w^{-1}$ and thus the operator $\left(
I + \frac{\Delta t^2}{\epsilon^2}L_w^{-1}
\right)$ is also positive definite and, in particular, invertible. 

Using the eq. \eqref{eq:q_updt_star}, we further have
\[
\left(
I + \frac{\Delta t}{\epsilon}\dfrac{1}{\mathcal M(u^n)}
- \frac{\Delta t}{\epsilon}r\,\Delta_h
\right)\mathbf{q}^*
+
\frac{\Delta t}{\epsilon}(1-\epsilon)\nabla_h^+\psi^*
=
\mathbf{q}^n - \frac{\Delta t}{\epsilon}(1-\epsilon)\nabla_h^+ \Pi(u^n).
\]
Thus, we have a coupled system for $(\psi^*, \mathbf{q}^*)$, which can be written as a linear system of equations of the form
\begin{equation}
\label{eq:lin_sys_psi_q}
    \begin{pmatrix}
    A_{\psi^*} & \dfrac{\Delta t}{\epsilon}(1-\epsilon)\,\nabla_h^- \cdot \\
    \dfrac{\Delta t}{\epsilon}(1-\epsilon)\,\nabla_h^+ & A_{\mathbf{q}^*}
    \end{pmatrix}
    \begin{pmatrix}
    \psi^* \\
    \mathbf{q}^*
    \end{pmatrix}
    =
    \begin{pmatrix}
    f_{\psi^n} \\
    f_{\mathbf{q}^n}
    \end{pmatrix},
\end{equation}
where
\begin{equation}
    A_{\psi^*} = I + \frac{\Delta t^2}{\epsilon^2} L_w^{-1},
\qquad A_{\mathbf{q}^*} = I + \frac{\Delta t}{\epsilon}\frac{1}{\mathcal{M}(u^n)} - \frac{\Delta t}{\epsilon} r\,\Delta_h,
\end{equation}
and
\begin{equation}
    f_{\psi^n} =
\psi^n
- \frac{\Delta t}{\epsilon} L_w^{-1}
\left(
w^n + \frac{\Delta t}{\epsilon}\gamma(1-\epsilon)\nabla_h^- \cdot \mathbf{p}^n
\right), \qquad f_{\mathbf{q}^n} =
 \mathbf{q}^n
- \frac{\Delta t}{\epsilon}(1-\epsilon)\nabla_h^+ \Pi(u^n).
\end{equation}
Once $(\psi^*, \mathbf{q}^*)$ is known by solving the linear system \eqref{eq:lin_sys_psi_q}, $w^*$ can be explicitly evaluated from \eqref{eq:w_updt_star}.

Finally, $\mathbf{U}_{\mathbf{j}}^{n+1}=(u_{\mathbf{j}}^{n+1},\psi_{\mathbf{j}}^{n+1},\mathbf{q}_{\mathbf{j}}^{n+1},w_{\mathbf{j}}^{n+1},\mathbf{p}_{\mathbf{j}}^{n+1})$ can be explicitly evaluated from the correction step \eqref{eq:expl} using the known predictor step solution $\mathbf{U}_{\mathbf{j}}^*=(u_{\mathbf{j}}^{*},\psi^*_{\mathbf{j}},\mathbf{q}_{\mathbf{j}}^{*},w_{\mathbf{j}}^{*},\mathbf{p}_{\mathbf{j}}^{*})$.

\begin{remark}
    Note that the predictor step \eqref{eq:semi-imp} of the numerical scheme is linearly implicit and the corrector step \eqref{eq:expl} is fully explicit due to carefully designed IMEX splitting between the stiff and non-stiff terms. The numerical solution from $t=t^n$ to $t=t^{n+1}$ is obtained using the scheme \eqref{eq:semi-imp}-\eqref{eq:expl} by solving exactly one linear system, and the rest of the components are evaluated explicitly.
\end{remark}

\section{Numerical results}\label{sec:numerics}
In this section, we present a series of numerical tests for various choices of mobility functions and disjoining pressures to validate the numerical scheme developed in Section \ref{sec: scheme}. To validate the accuracy and scope of the scheme, we compare the numerical solutions obtained by our numerical scheme with those of the second-order finite difference scheme developed in the seminal work of Zhornitskaya and Bertozzi \cite{zhornitskaya1999positivity}. In all our test cases, we refer to the solutions obtained using the finite difference scheme \cite{zhornitskaya1999positivity} as a reference solution. For the sake of simplicity, we keep our attention in this article to one-dimensional test cases and with periodic boundary conditions and $\gamma=1$, but similar results can be obtained also for multi-dimensional test cases.
\subsection{Smooth positive solutions of thin film equation \eqref{eq: main}}
\label{sec:1d_first}
In this example, we consider a test case for which the thin film equation \eqref{eq: main} possesses a uniformly positive solution $u$ for positive initial data $u_0$ \cite{bernis1990higher}. In particular, we consider $\Pi(u)=0$ and $\mathcal{M}(u)=u^4$, and the following periodic initial data
\begin{equation}
\begin{aligned}
    u^\epsilon(0,x) &= 0.8 - \cos((x-\pi )) + 0.25  \cos(2 (x -\pi)),\\
     \psi^\epsilon(0,x) &= -\gamma\partial_{xx} u_0^\epsilon,\\ q^\epsilon(0,x) &=
    \mathcal{M}(u_0^\epsilon)\partial_x\big(\gamma \partial_{xx} u_0^\epsilon - \Pi(u_0^\epsilon)\big),\\
     w^\epsilon(0,x) &= -\partial_x\big(
    \mathcal{M}(u_0^\epsilon)\partial_x\big(\gamma \partial_{xx} u_0^\epsilon - \Pi(u_0^\epsilon)\big)\big),\\
    p^\epsilon(0,x) &= \partial_x u_0^\epsilon.
\end{aligned}
\end{equation}
The computational domain $\Omega=[0,2\pi]$ is divided into $2000$ mesh points for the proposed scheme \eqref{eq:semi-imp}-\eqref{eq:expl}. In Figure~\ref{fig:first_test}, we depict the solution profiles at $T=0.01$ and the corresponding initial values for different choices of $\epsilon$. To validate the accuracy of our numerical approach, we compare our numerical solutions with the reference solution obtained by the positivity-preserving numerical scheme \cite{zhornitskaya1999positivity} with $100$ mesh points. The numerical solution of the relaxation system obtained using \eqref{eq:semi-imp}-\eqref{eq:expl} approximates the limit solution remarkably well, even though our proposed scheme is only first-order accurate.

\begin{figure}[htbp]
  \centering
    \includegraphics[height=0.185\textheight]{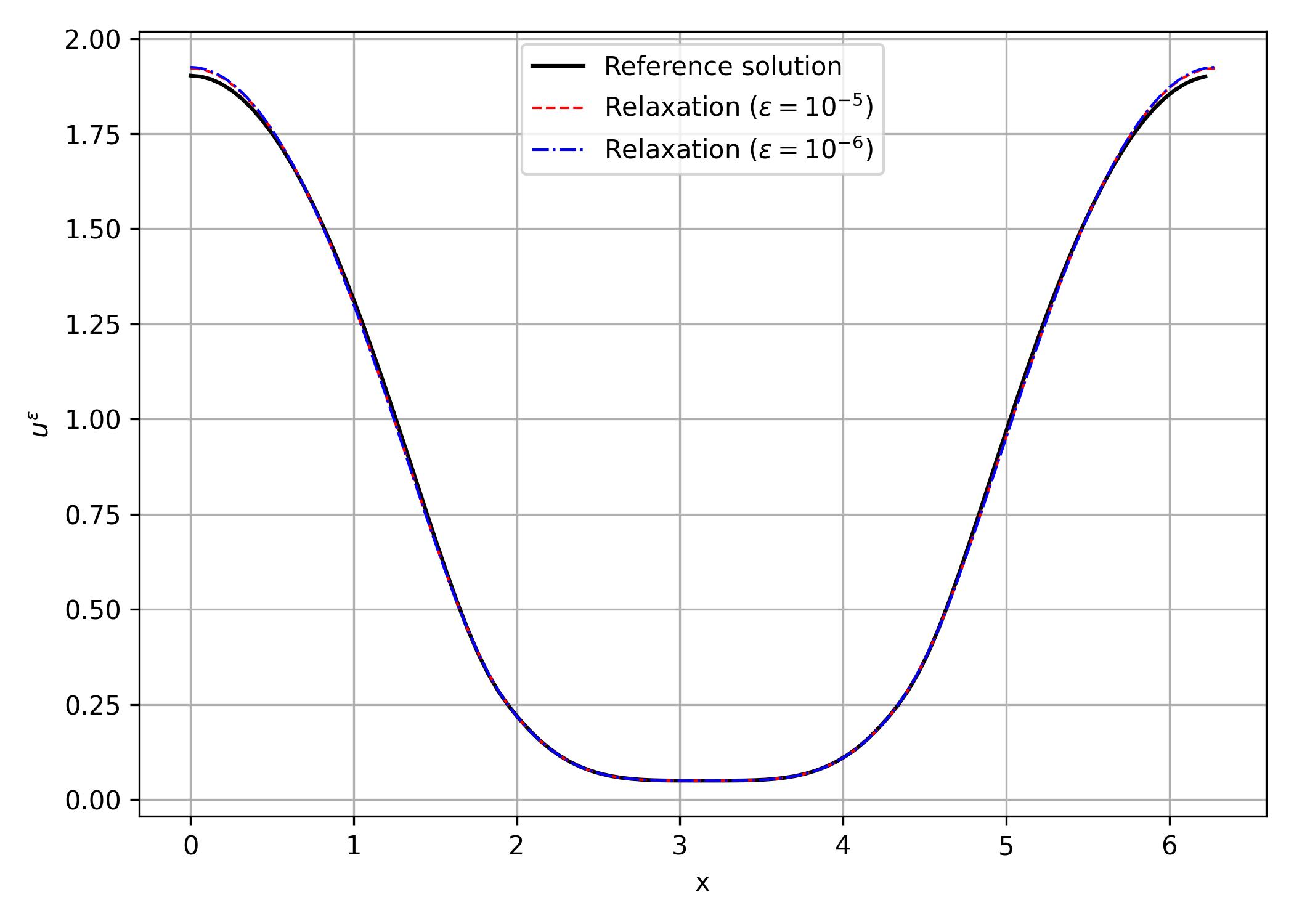}
    \includegraphics[height=0.185\textheight]{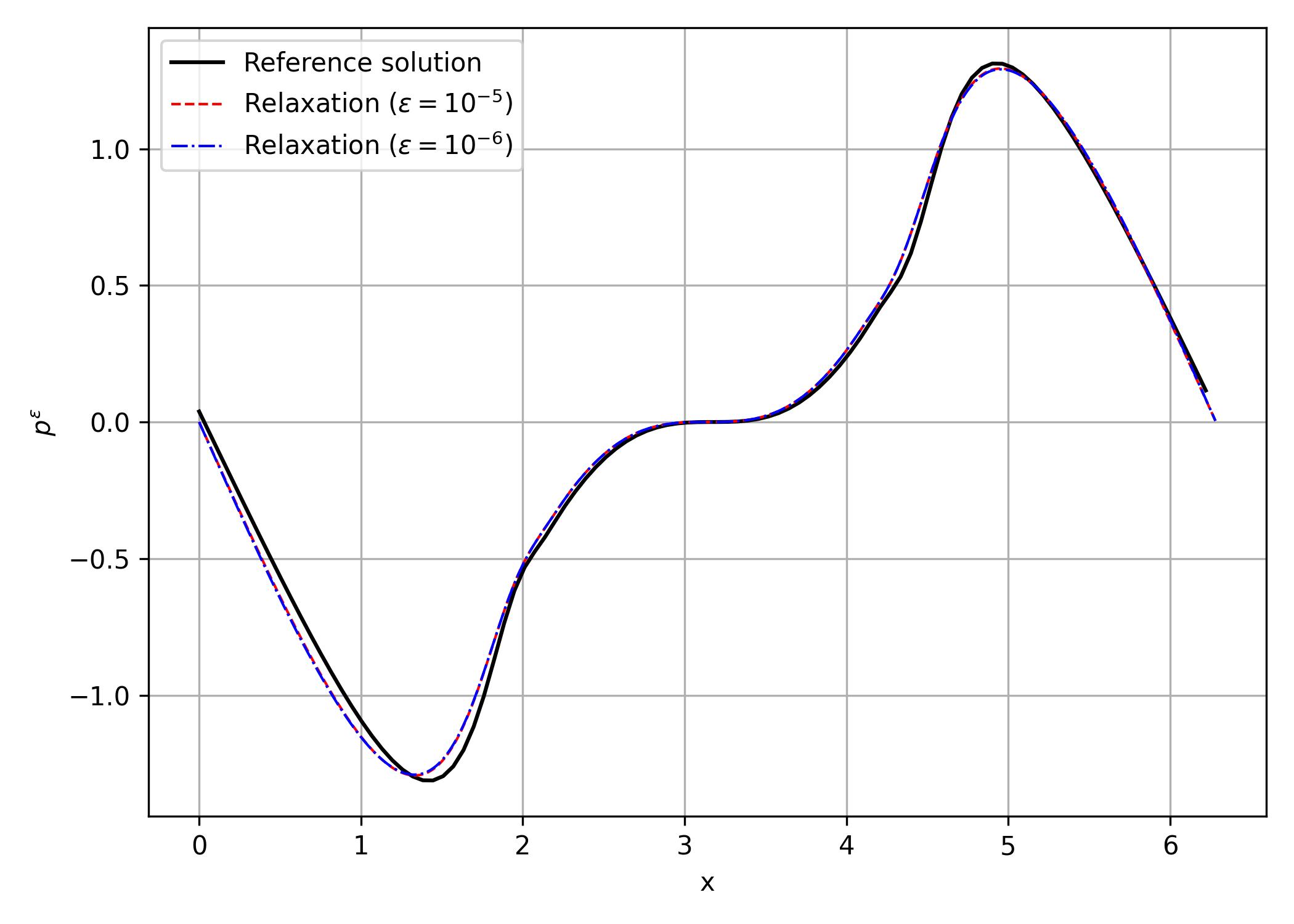}
    \includegraphics[height=0.185\textheight]{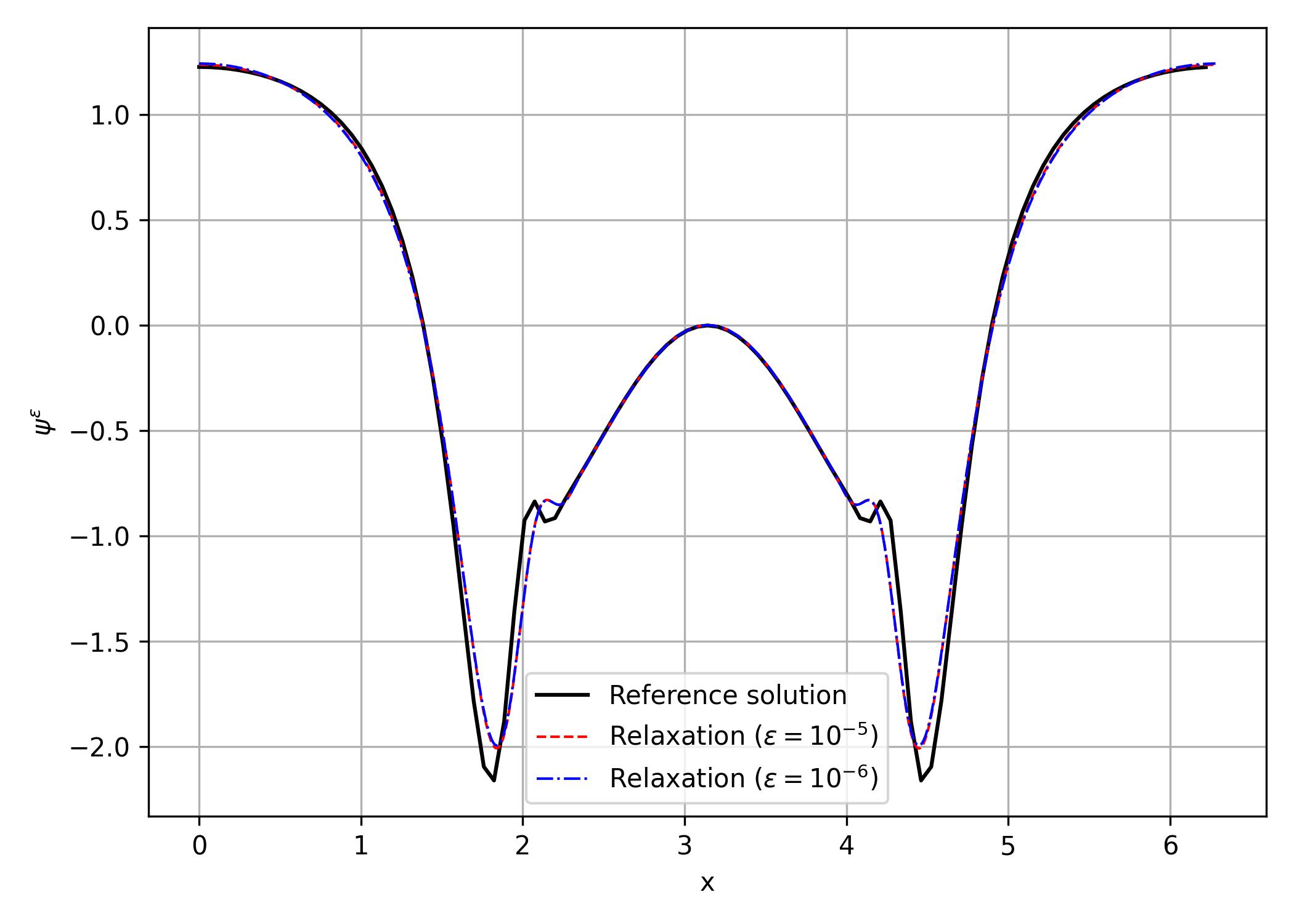}
    \includegraphics[height=0.185\textheight]{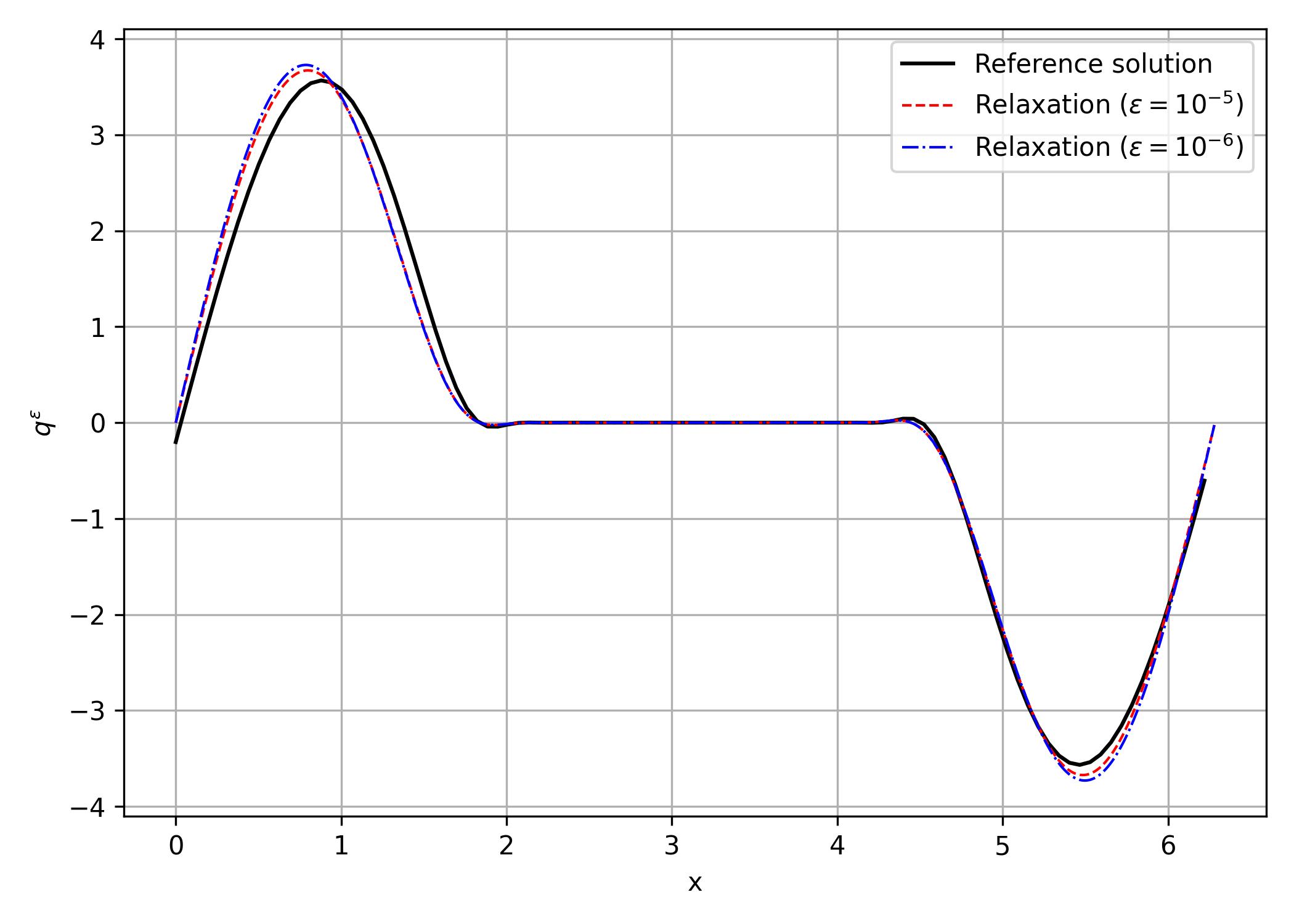}
    \includegraphics[height=0.185\textheight]{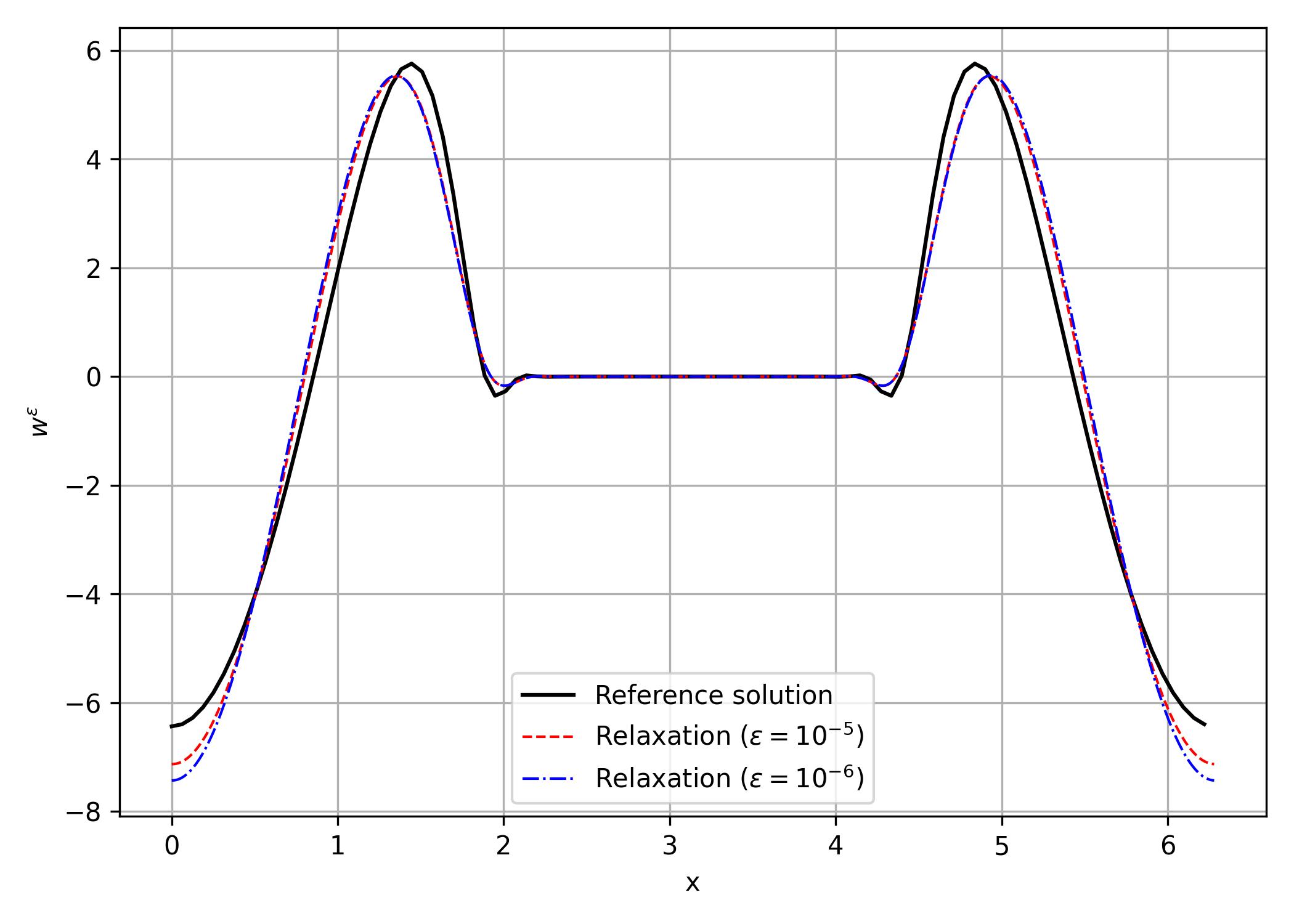}
    \caption{Comparison of numerical solutions and reference solutions at $T=0.01$ for different relaxation parameter values for the test case \ref{sec:1d_first}.}  
    \label{fig:first_test}
\end{figure}

\subsection{Positivity preservation for degenerate mobility}
\label{sec:1d_second}
In this example, we consider the test case introduced by Zhornitskaya and Bertozzi \cite{zhornitskaya1999positivity}. Specifically, we set $\mathcal{M}(u)=\sqrt{u}$, and $\Pi(u)=0$. To run simulations for this test case, a positivity-preserving numerical scheme is necessary (see \cite{zhornitskaya1999positivity} or \cite{grun2000nonnegativity} for more details). We regularize the mobility to obtain a positive numerical solution along with Lemma \ref{lemma: positivity}. Precisely, for $\delta>0$, we regularize the mobility function following \cite{zhornitskaya1999positivity} as follows
\[
\mathcal{M}_\delta(u_\delta)=\dfrac{u_\delta^4 \mathcal{M}(u_\delta)}{\delta \mathcal{M}(u_\delta)+u_\delta^4}.
\]
Assuming well-prepared initial data of the form \eqref{initial_data_hyp_relaxation}, we define the following initial conditions for the unknown vector $\mathbf{U} = (u, \psi, q, w, p)^\top$:
\begin{equation}
\begin{aligned}
    u^\epsilon(0,x) &= 0.8 - \cos(\pi x) + 0.25 \cos(2 \pi x) =: u_0^\epsilon,\\
    \psi^\epsilon(0,x) &= -\gamma\partial_{xx} u_0^\epsilon,\\
     q^\epsilon(0,x) &= \sqrt{u_0^\epsilon}\partial_x\big(\gamma \partial_{xx} u_0^\epsilon\big),\\
       w^\epsilon(0,x) &= -\partial_x\Big( \sqrt{u_0^\epsilon}\partial_x\big(\gamma \partial_{xx} u_0^\epsilon - \Pi(u_0^\epsilon)\big)\Big),\\
    p^\epsilon(0,x) &= \partial_x u_0^\epsilon.
  \end{aligned}
\end{equation}
We discretize the computational domain $\Omega=[-1,1]$ using $4000$ mesh points for the numerical scheme \eqref{eq:semi-imp}-\eqref{eq:expl}. Figure~\ref{fig:second_test} illustrates the solution profiles for the variable $u^\eps$ at various times alongside its corresponding initial values, evaluated with a relaxation parameter $\epsilon=10^{-6}$. We compare the numerical solutions obtained using the proposed scheme against a reference solution obtained using the finite difference scheme from \cite{zhornitskaya1999positivity}. The numerical solution of the relaxation system approximates the limit equation's solution remarkably well, even though our proposed scheme is only first-order accurate. Furthermore, Figure~\ref{fig:second_test_energy} displays the energy profiles of the numerical solutions for the system \eqref{hyperbolic_system} with $\epsilon=10^{-5}$ and $\epsilon=10^{-6}$, comparing them directly against the energy profile of the limit equation \eqref{eq: main}. The asymptotic behavior of these profiles demonstrates the energy consistency of our approach.

A key outcome of this test is that it confirms the positivity-preserving nature of the numerical solution for the relaxation system \eqref{hyperbolic_system} obtained using the scheme \eqref{eq:semi-imp}-\eqref{eq:expl}. As pointed out in \cite{zhornitskaya1999positivity}, numerical evaluations for this specific test case can yield spurious results if the grid is insufficiently resolved. Our successful approximation demonstrates that the scheme developed in this article is both asymptotic-preserving and positivity-preserving.
\begin{figure}[htbp]
  \centering
    \includegraphics[height=0.25\textheight]{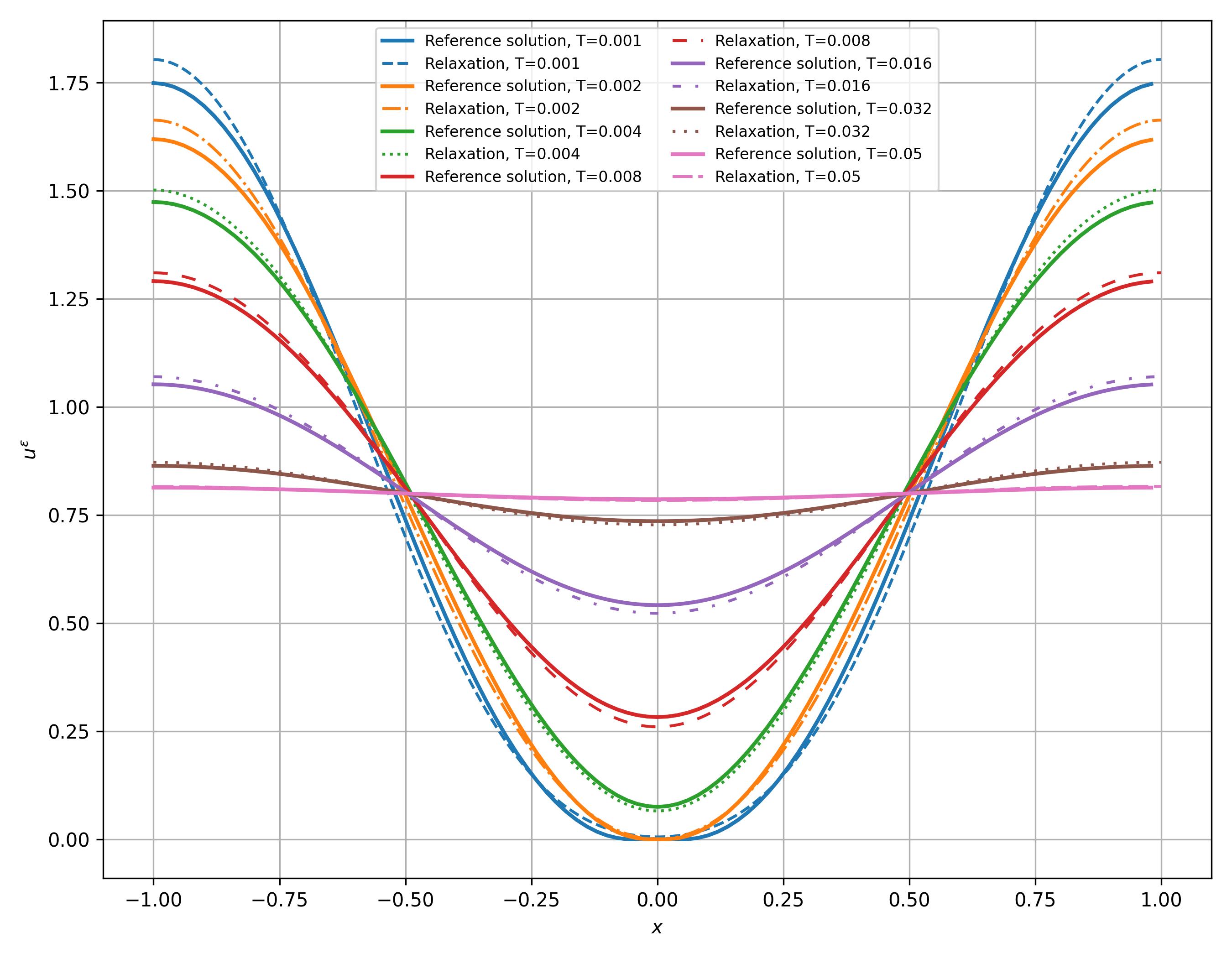}
    \caption{Comparison of numerical solutions for the relaxation system \eqref{hyperbolic_system} and limit equation \eqref{eq: main} at different times corresponding to relaxation parameter $\epsilon=10^{-6}$ for the test case \ref{sec:1d_second}.}  
    \label{fig:second_test}
\end{figure}

\begin{figure}[htbp]
  \centering
    \includegraphics[height=0.25\textheight]{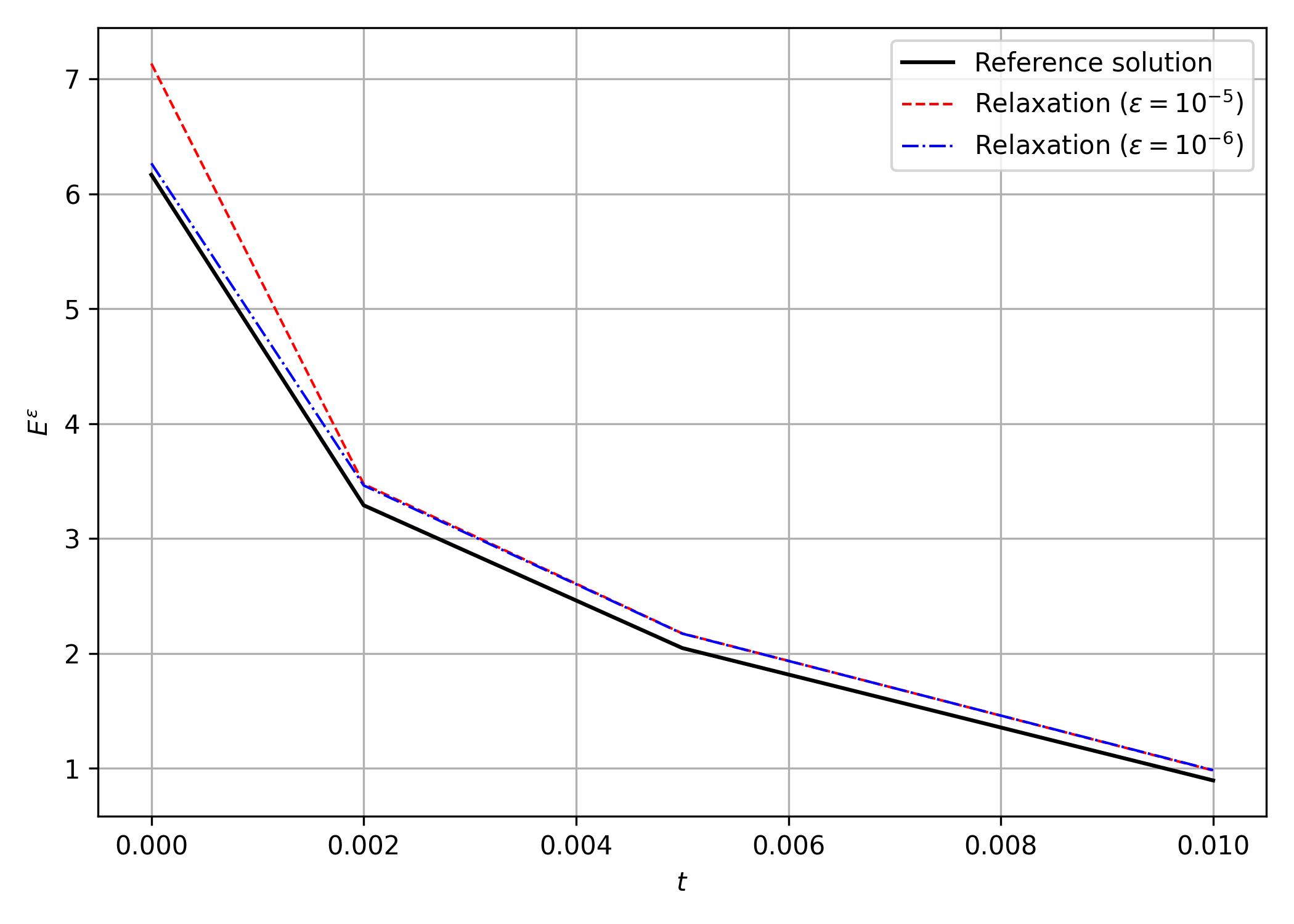}
    \caption{Comparison of energy plots corresponding to small values of relaxation parameters for the test case \ref{sec:1d_second}.}  
    \label{fig:second_test_energy}
\end{figure}
\subsection {Near-rupture solution of a thin film equation}
\label{sec:1d_third}
In this test case, we consider the Hele-Shaw law for the thin-film equation and set $\Pi(u)=0$ and $\mathcal{M}(u)=u$. We utilize the following well-prepared initial data in this test problem. The initial data is inspired by a test case in \cite{becker2005thin}.  
\begin{equation}
\begin{aligned}
    u^\epsilon(0,x) &= (x - 0.5)^4 + 0.001,\\
     \psi^\epsilon(0,x) &= -\gamma\partial_{xx} u_0^\epsilon,\\ q^\epsilon(0,x) &=
    \mathcal{M}(u_0^\epsilon)\partial_x\big(\gamma \partial_{xx} u_0^\epsilon - \Pi(u_0^\epsilon)\big),\\
    w^\epsilon(0,x) &= -\partial_x\big(
    \mathcal{M}(u_0^\epsilon)\partial_x\big(\gamma \partial_{xx} u_0^\epsilon - \Pi(u_0^\epsilon)\big)\big),\\
    p^\epsilon(0,x) &= \partial_x u_0^\epsilon.
\end{aligned}
\end{equation}
We divide the computational domain $\Omega=[0,1]$ into $4000$ mesh points for the proposed scheme \eqref{eq:semi-imp}-\eqref{eq:expl}. To gain insights into the near-rupture regime and behavior of our proposed scheme \eqref{eq:semi-imp}-\eqref{eq:expl}, we plot the solution profiles for the conserved variable $u^\eps$ in Figure~\ref{fig:third_test} at times $T=0.00, 0.002, 0.012,$ and $0.04$ with different choices of $\epsilon$. We compare our numerical solutions with the reference solution obtained using the positivity-preserving finite difference scheme developed in \cite{zhornitskaya1999positivity}. One can clearly observe that at time $T=0.012$, the solution $u^\eps$ tends to go very near the rupture regime. However, the positivity preservation proposed in our scheme handles it precisely, and the simulation does not break after $T=0.012$. After this critical time, the numerical solution becomes uniformly positive. The solution profile looks almost identical for both schemes. We also plot the energy profile of the proposed numerical scheme in Figure \ref{fig:third_test_energy}. Energy consistency is again observed regardless of the choice of the mobility function or the initial data.
\begin{figure}[htbp]
  \centering
    \includegraphics[height=0.185\textheight]{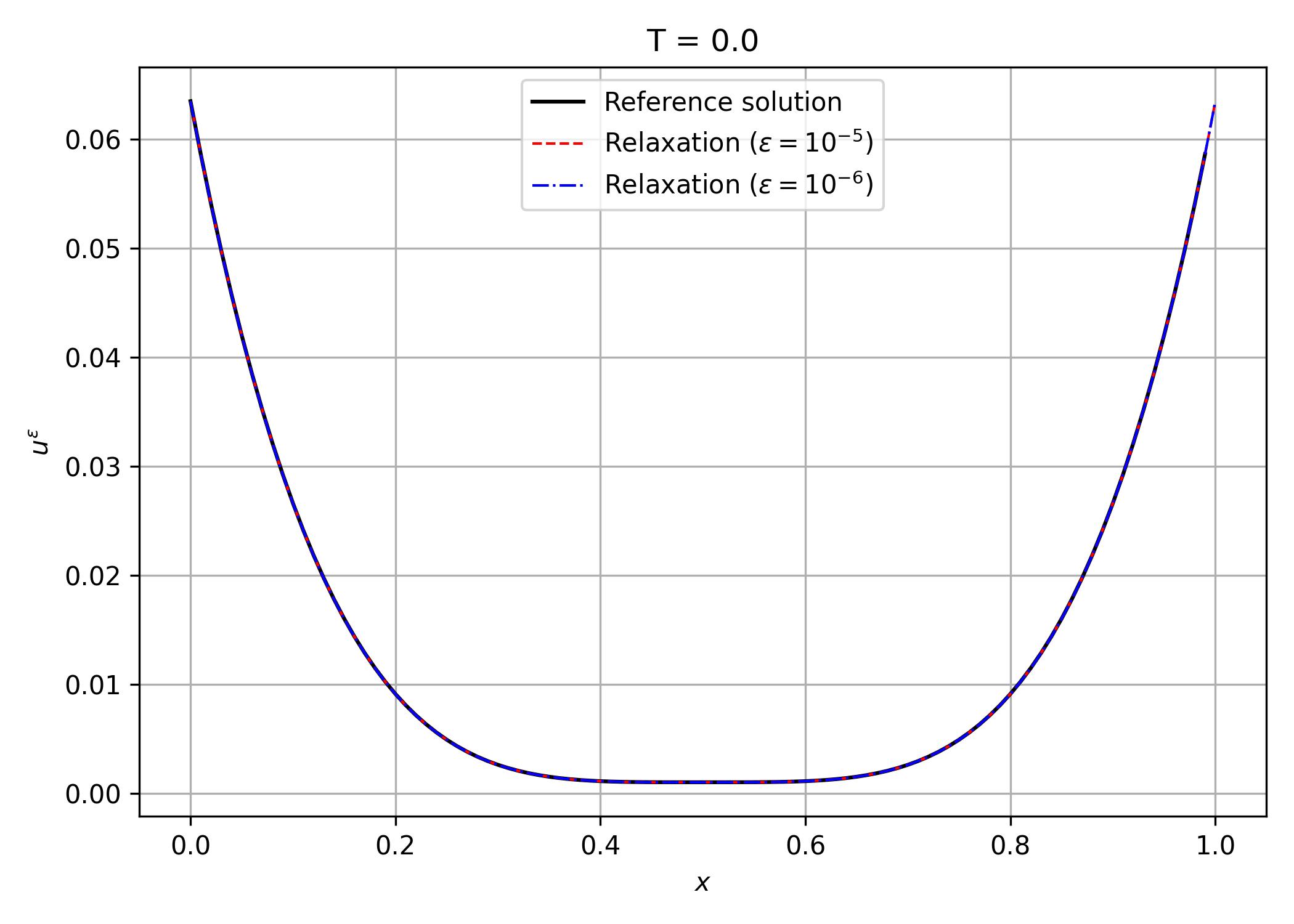}
    \includegraphics[height=0.185\textheight]{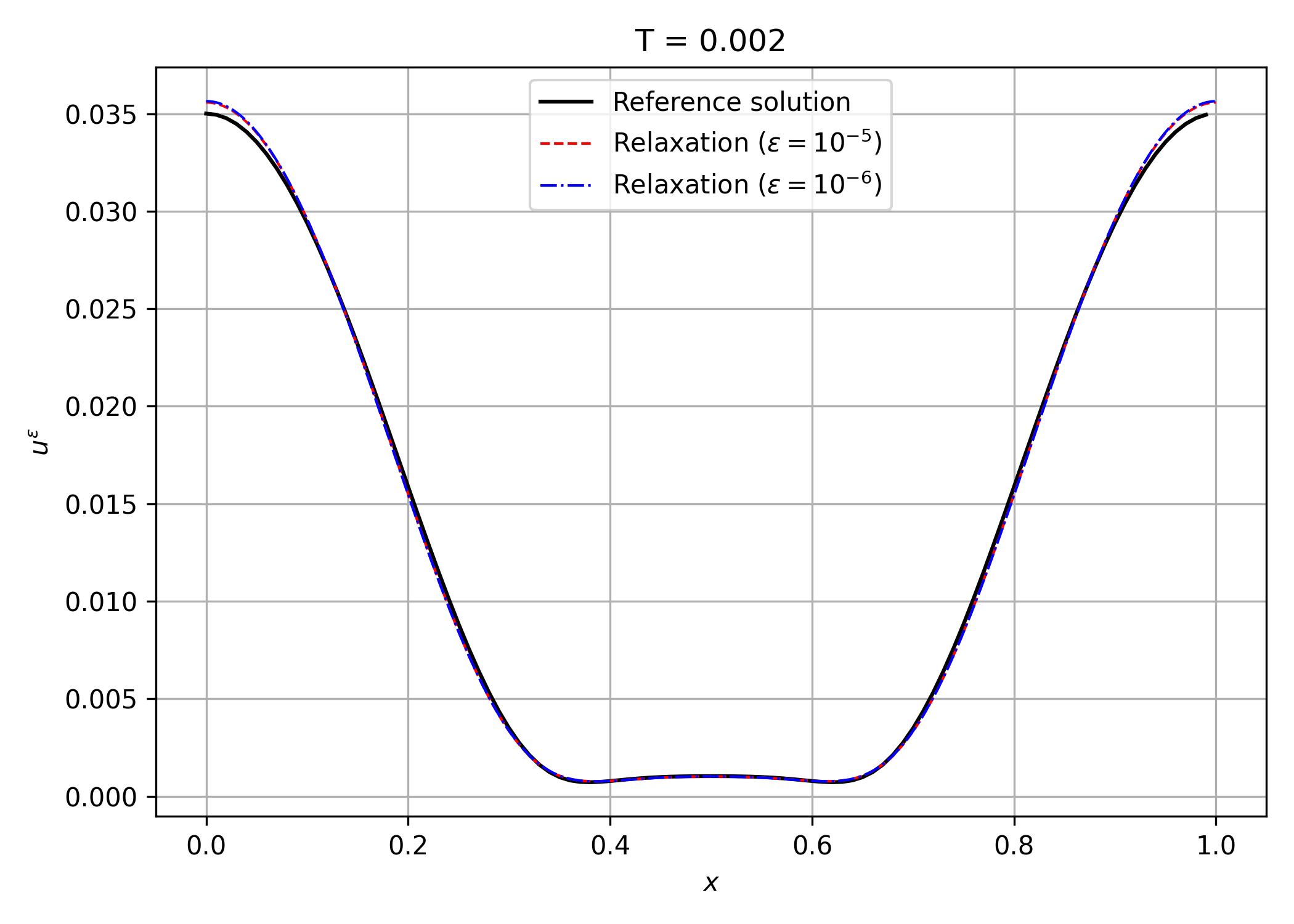}
    \includegraphics[height=0.185\textheight]{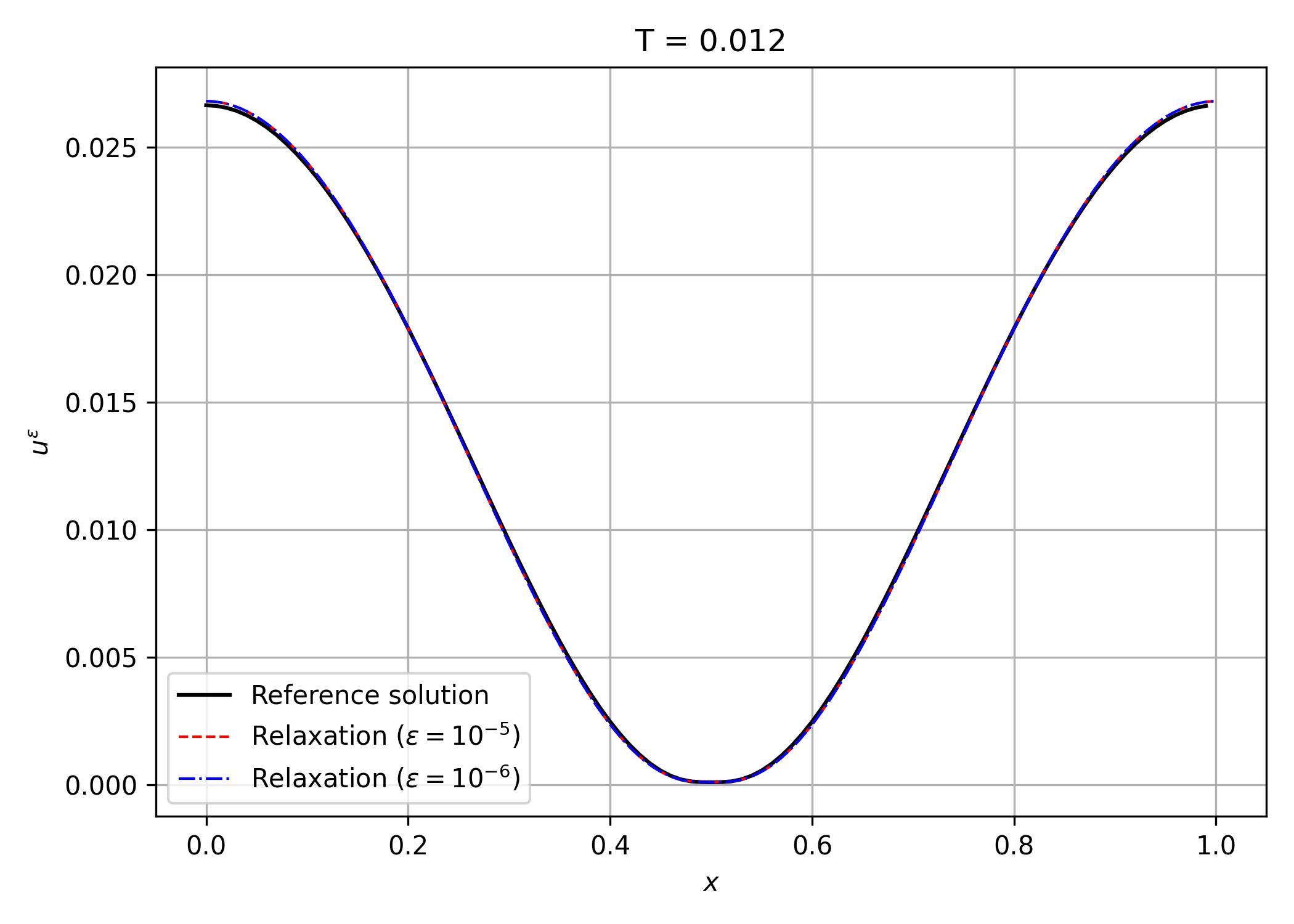}
    \includegraphics[height=0.185\textheight]{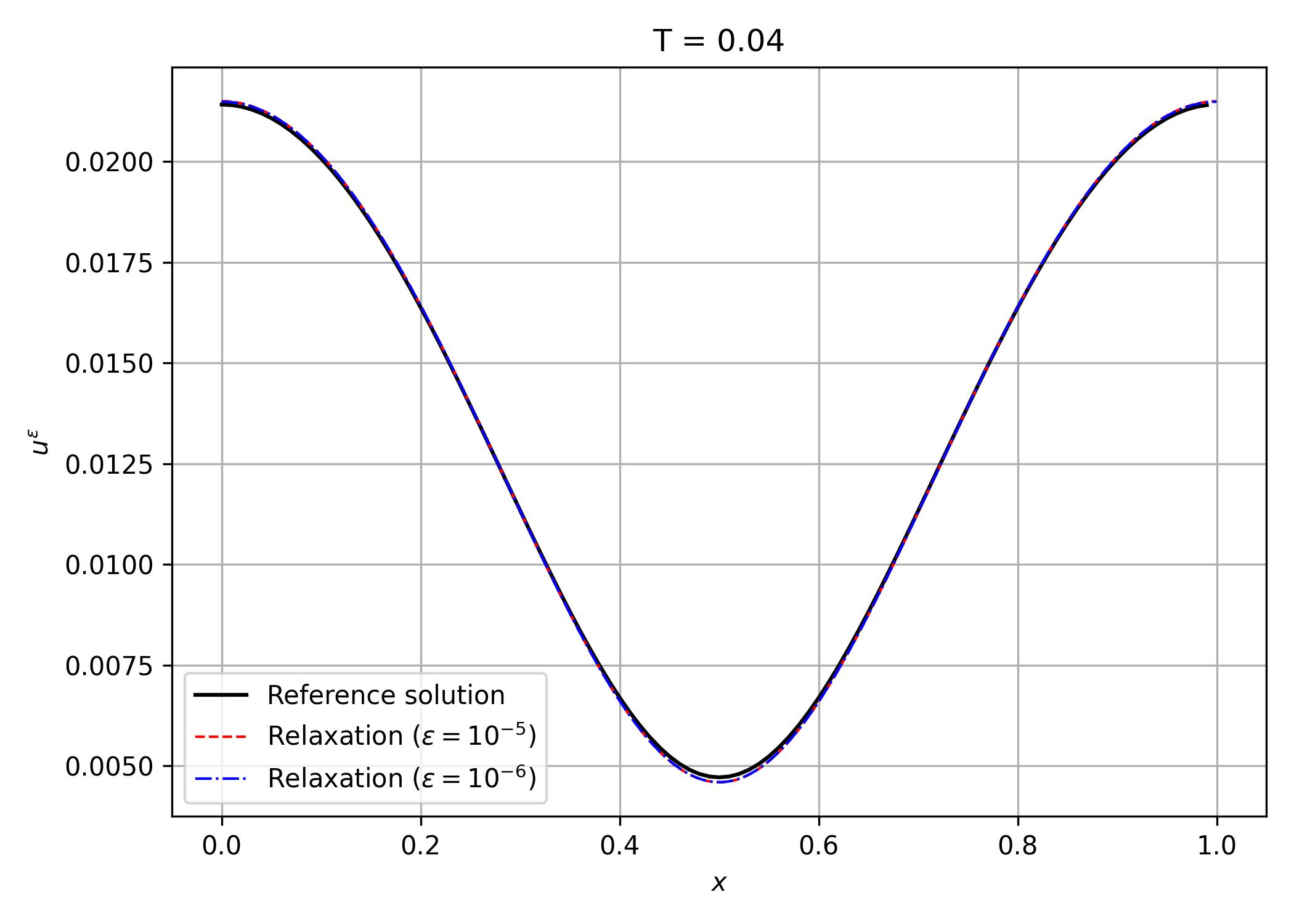}
    \caption{Comparison of temporal evolution of film heights obtained using the proposed scheme and the reference solution with different epsilon values for the test case \ref{sec:1d_third}.}  
    \label{fig:third_test}
\end{figure}

\begin{figure}[htbp]
  \centering
    \includegraphics[height=0.25\textheight]{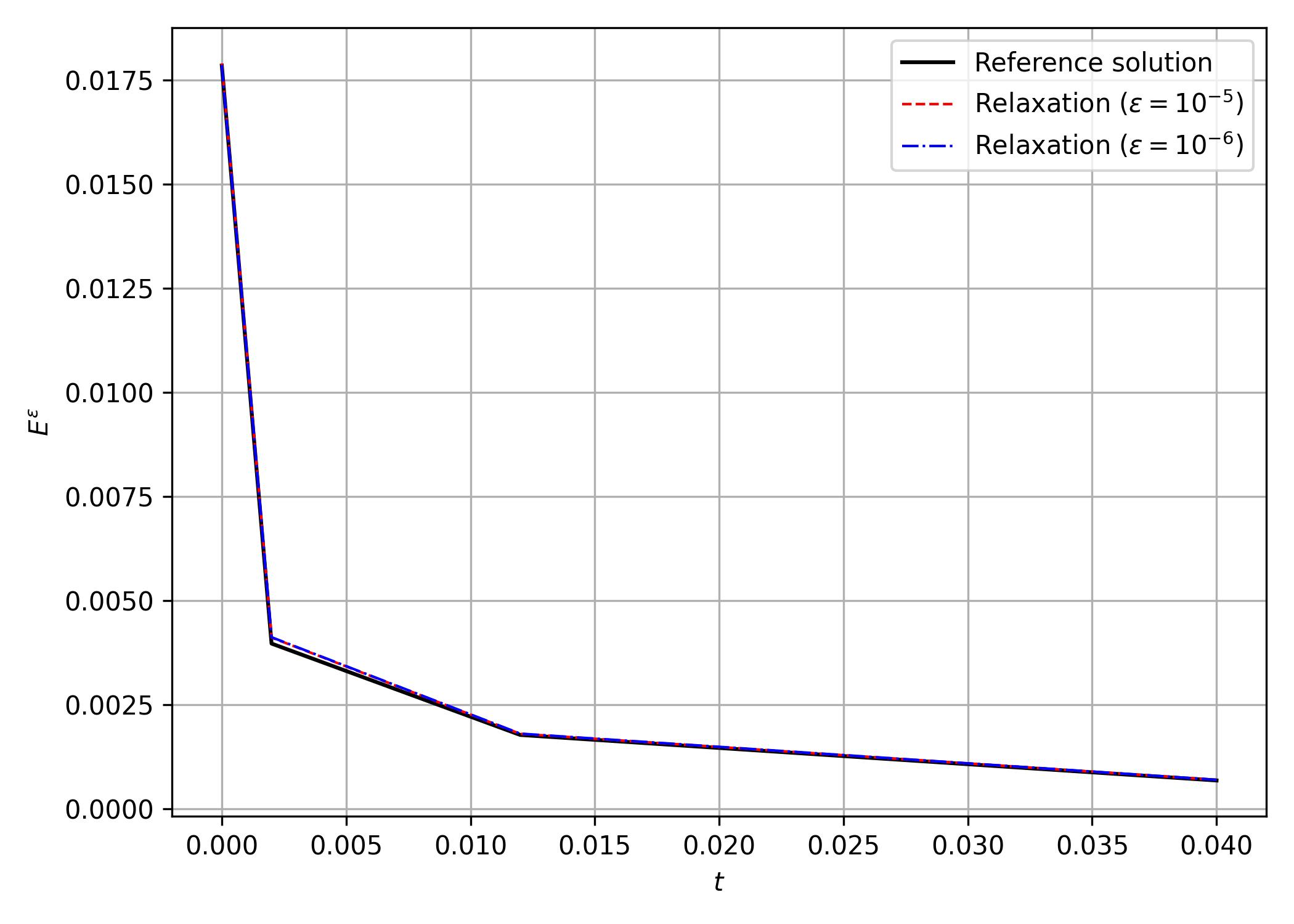}
    \caption{Comparison of energy plots corresponding to small values of the relaxation parameter for the test case \ref{sec:1d_third}.}  
    \label{fig:third_test_energy}
\end{figure}

\subsection{Constant equilibrium with nonzero pressure}
\label{sec:1d_fourth}
To investigate the effect of nonzero disjoining pressure on the solution profile, we consider the following initial data in this example.
\begin{equation}
\begin{aligned}
    u^\epsilon(0,x) &= 1.0 + 0.005 \sin(\pi x),\\
    \psi^\epsilon(0,x) &= -\gamma\partial_{xx} u_0^\epsilon,\\ q^\epsilon(0,x) &=
    \mathcal{M}(u_0^\epsilon)\partial_x\big(\gamma \partial_{xx} u_0^\epsilon - \Pi(u_0^\epsilon)\big),\\
    w^\epsilon(0,x) &= -\partial_x\big(
    \mathcal{M}(u_0^\epsilon)\partial_x\big(\gamma \partial_{xx} u_0^\epsilon - \Pi(u_0^\epsilon)\big)\big),\\
    p^\epsilon(0,x) &= \partial_x u_0^\epsilon,
\end{aligned}
\end{equation}
with $\Pi(u)= u$ and $\mathcal{M}(u)=u^3$.

The computational domain $\Omega=[0,2]$ is divided into $4000$ mesh points for the proposed scheme \eqref{eq:semi-imp}-\eqref{eq:expl}. We plot the numerical solutions for the unknown vector $\mathbf{U} = (u, \psi, q, w, p)^\top$ in the Figure \ref{fig:fourth_test} for times $T=0.00, 0.01$ and $0.1$ corresponding to $\epsilon=10^{-6}$. We also compare our numerical solutions with the reference solution obtained using the finite difference scheme proposed in \cite{zhornitskaya1999positivity} or \cite{kim2024positivity}. In this test case, the initial sinusoidal profile settles to a constant equilibrium after a finite time. Our scheme captures the constant equilibrium really well. Results obtained using the proposed first-order scheme are comparable with those of the higher-order finite difference scheme. Energy stability can be observed in Figure \ref{fig:fourth_test_energy}.
\begin{figure}[htbp]
  \centering
    \includegraphics[height=0.185\textheight]{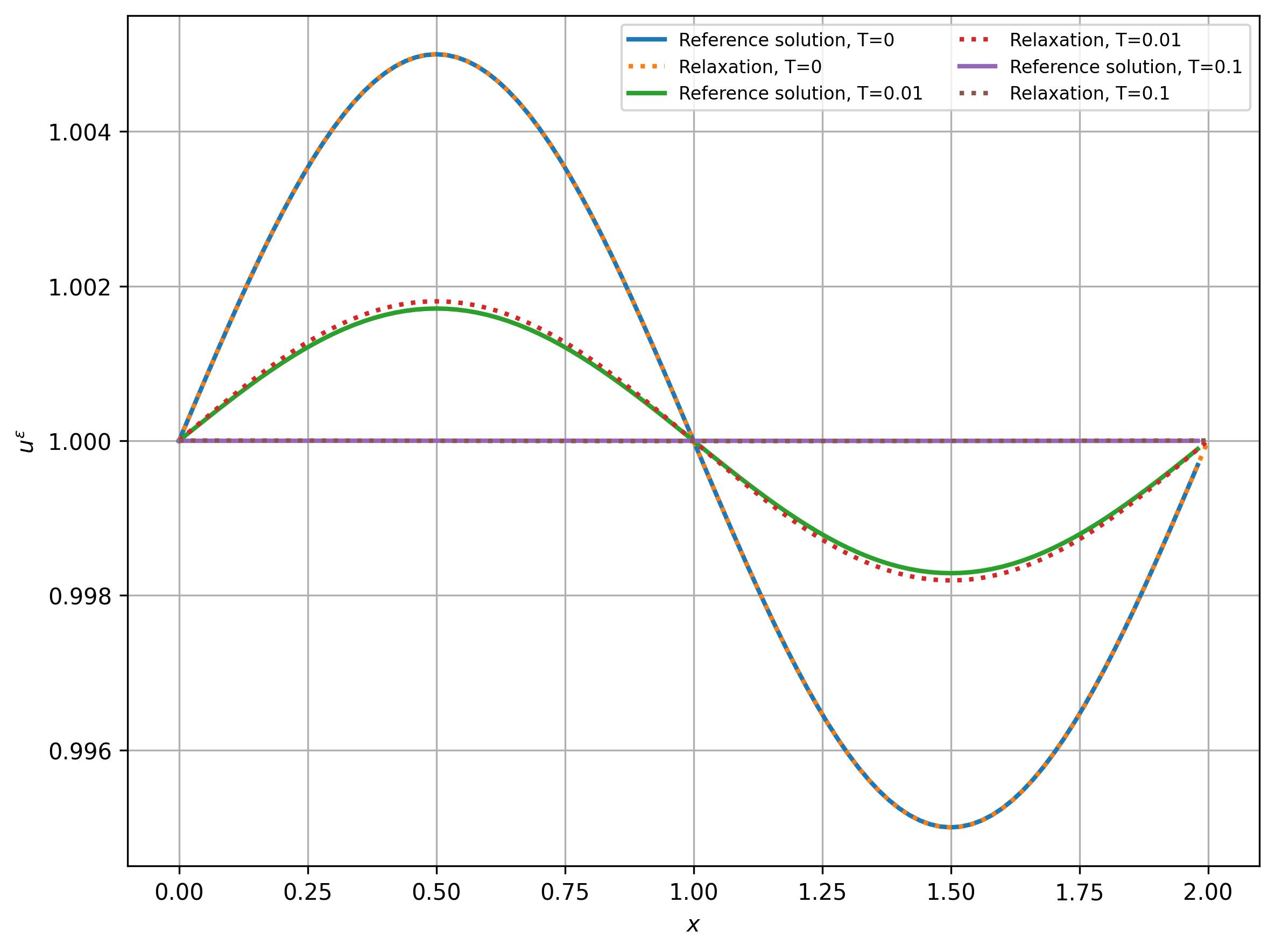}
    \includegraphics[height=0.185\textheight]{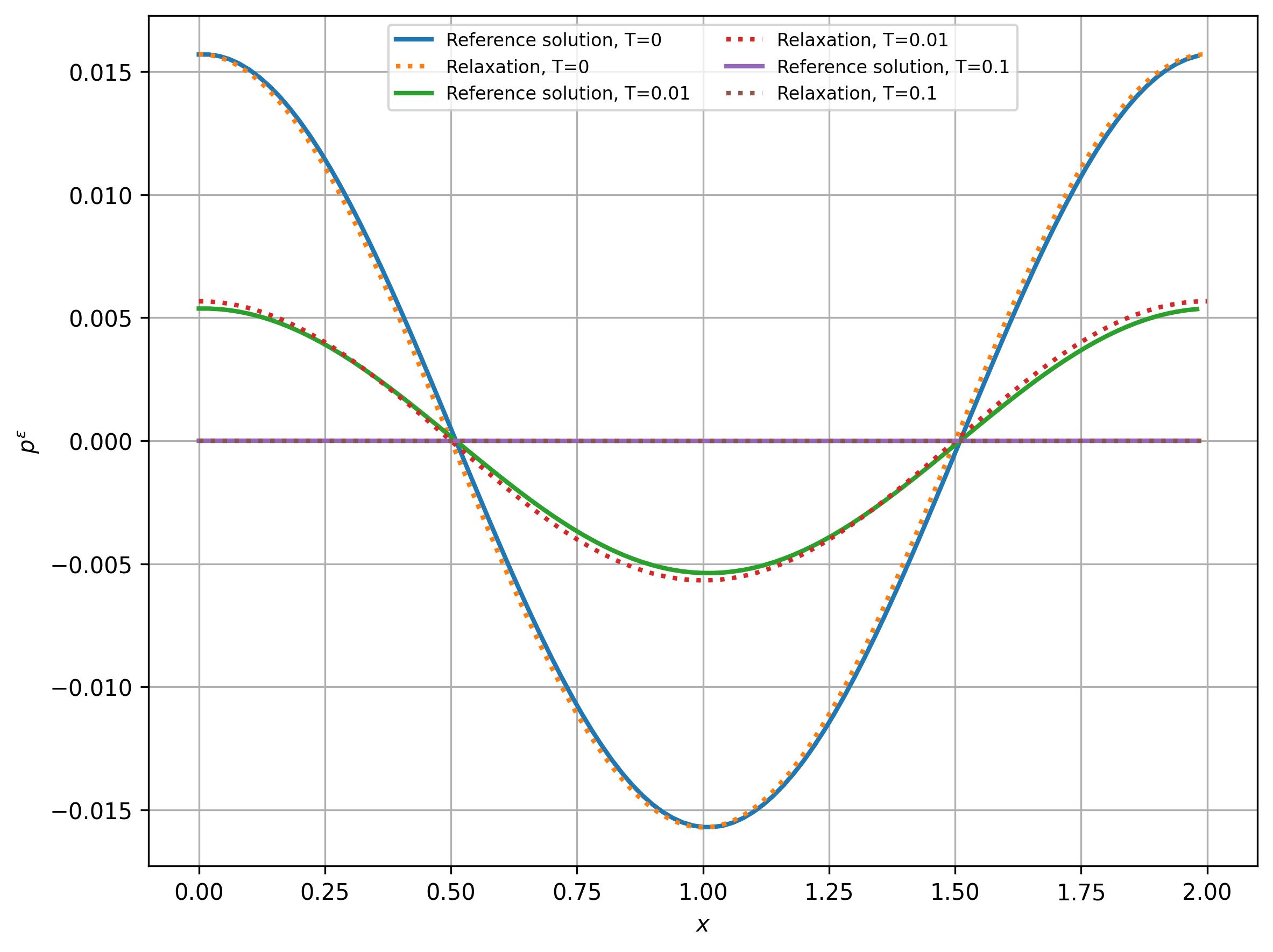}
    \includegraphics[height=0.185\textheight]{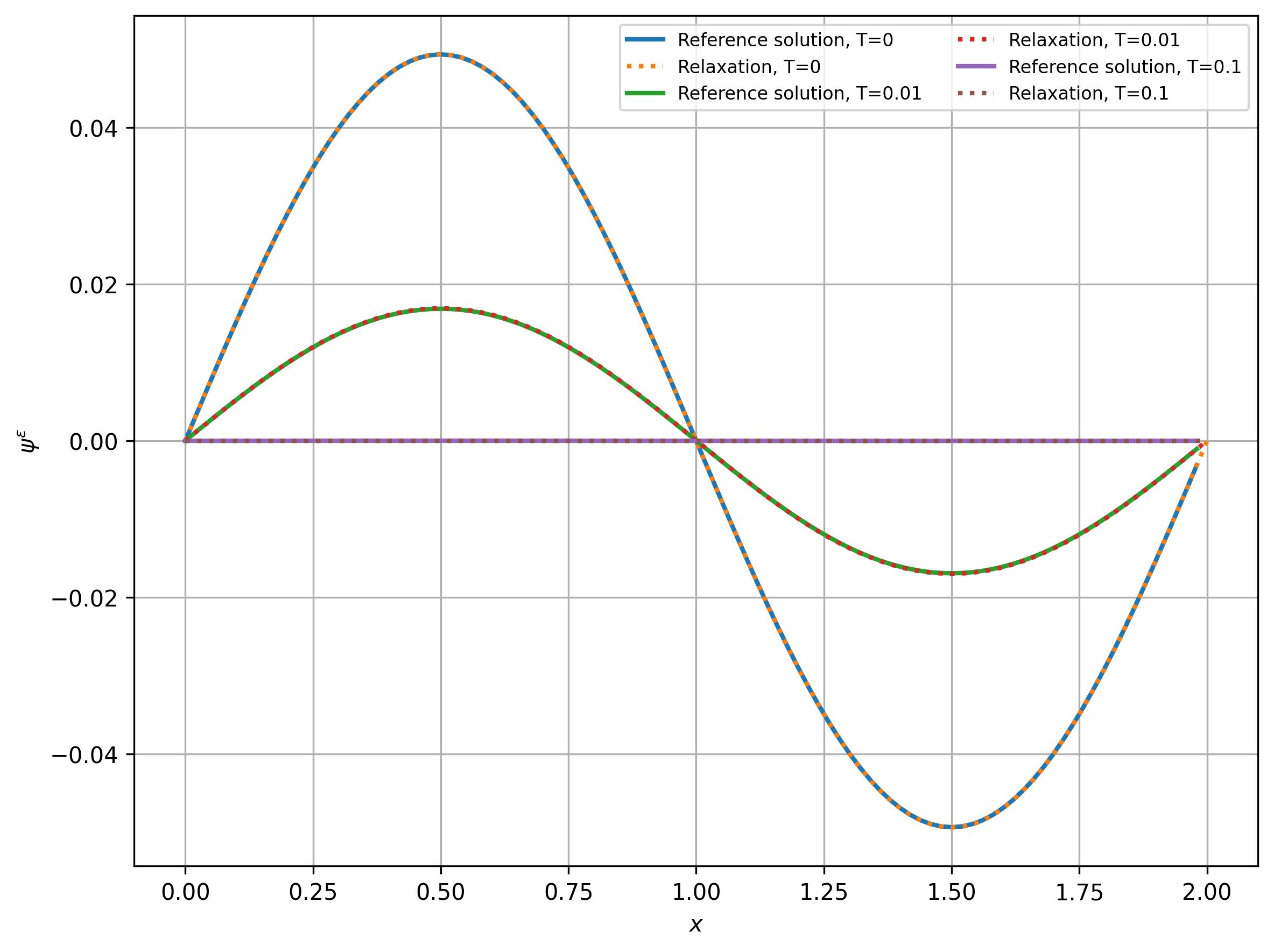}
    \includegraphics[height=0.185\textheight]{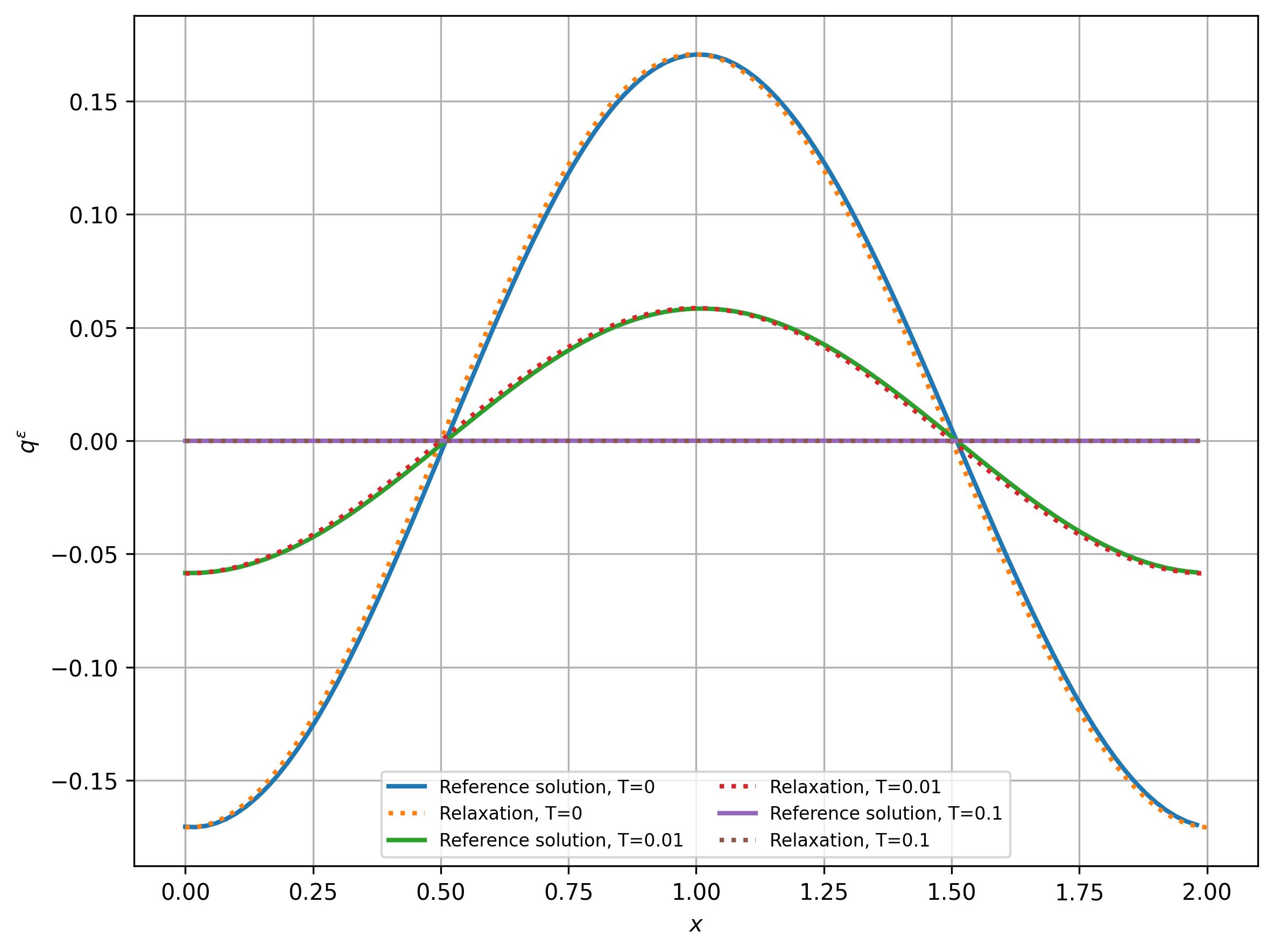}
    \includegraphics[height=0.185\textheight]{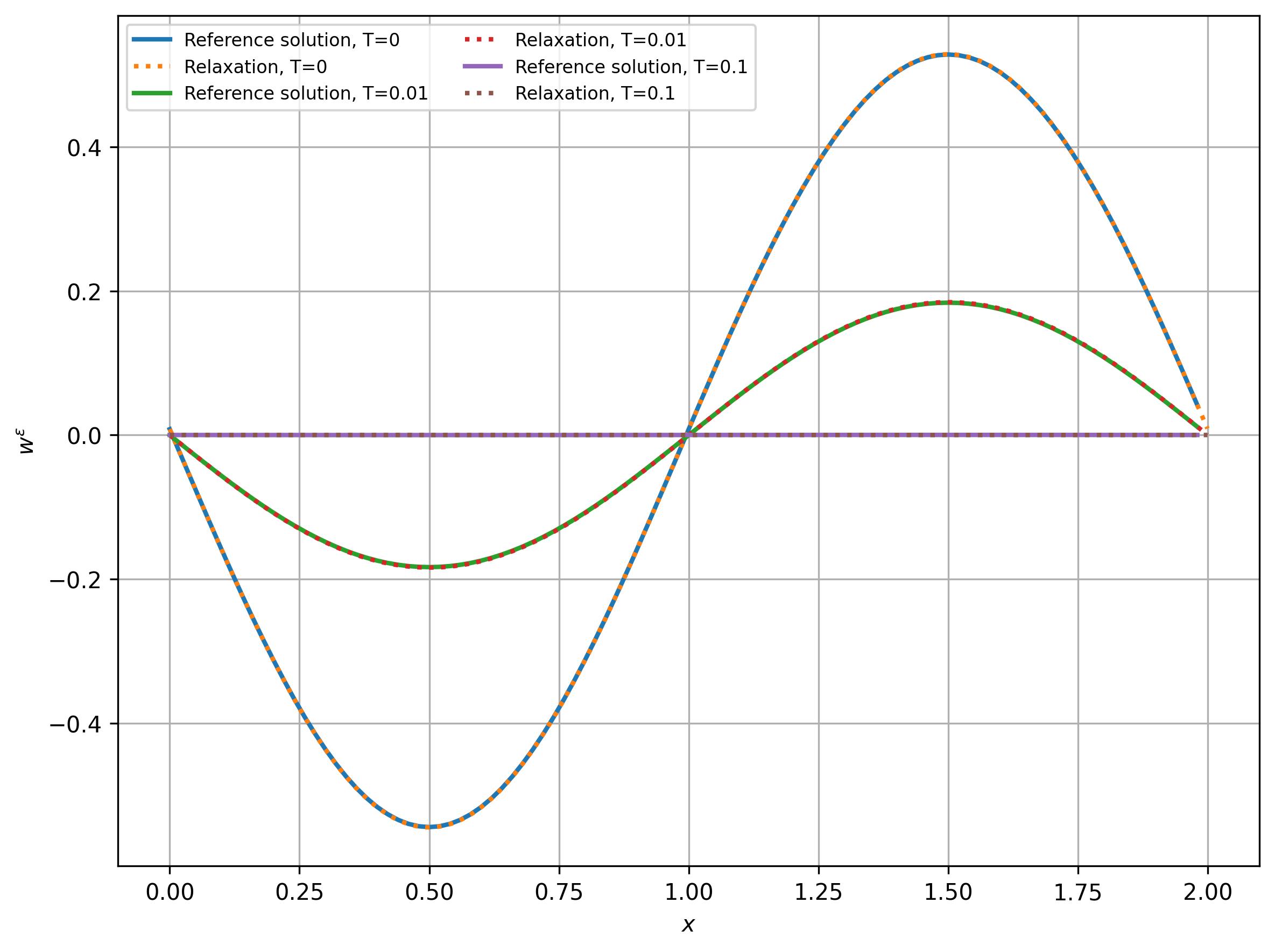}
    \caption{Comparison of numerical solutions obtained using proposed scheme with $\epsilon=10^{-6}$ and reference solutions at different times for the test case \ref{sec:1d_fourth}.}  
    \label{fig:fourth_test}
\end{figure}

\begin{figure}[htbp]
  \centering
    \includegraphics[height=0.25\textheight]{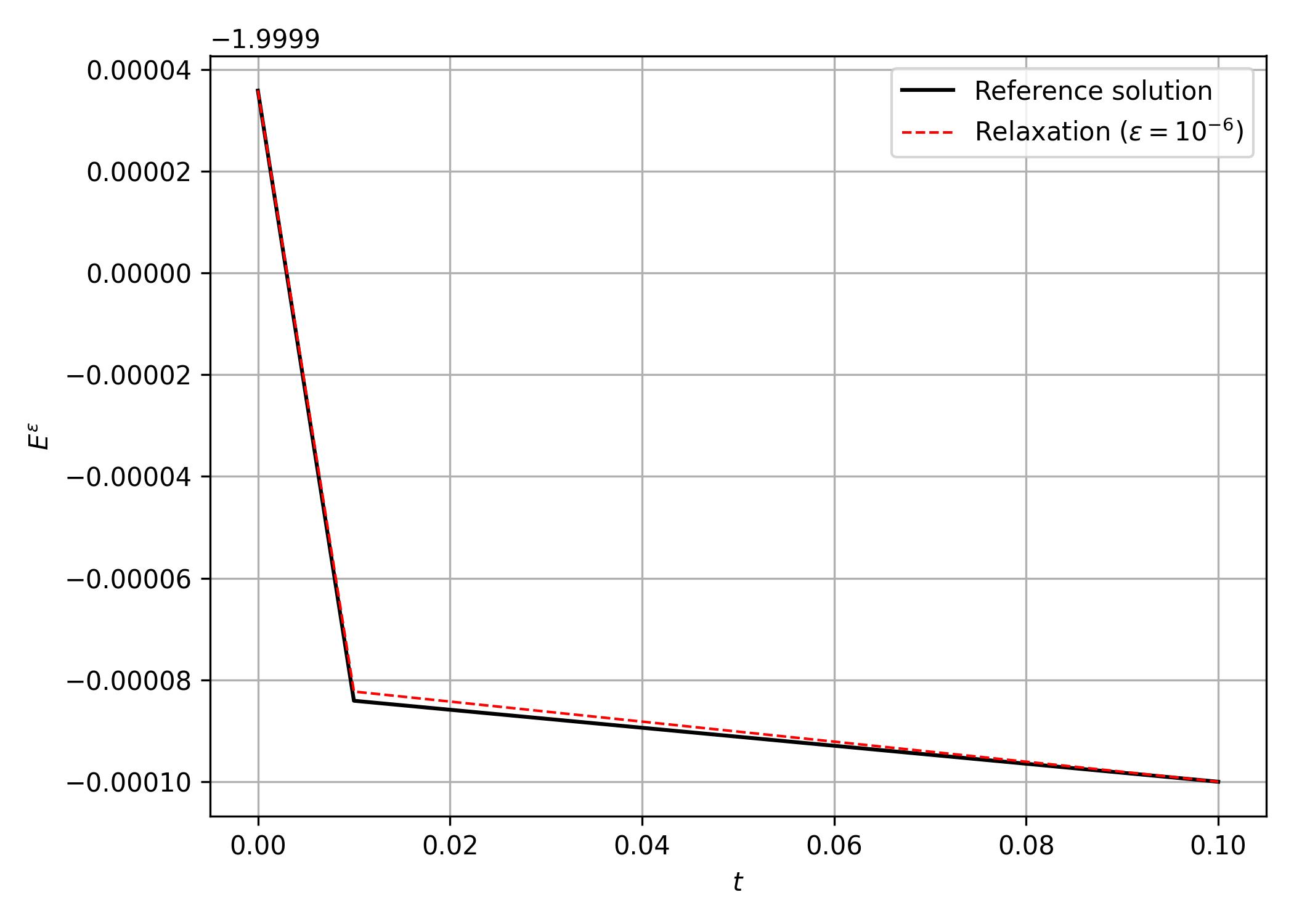}
    \caption{Comparison of energy plots corresponding to $\epsilon=10^{-6}$ for the test case \ref{sec:1d_fourth}.}  
    \label{fig:fourth_test_energy}
\end{figure}
\section{Conclusions and future outlook}\label{sec: conclusions}
In this article, we have proposed a novel energy-consistent, positivity-preserving, and asymptotic-preserving implicit-explicit (IMEX) scheme for a hyperbolic relaxation system that approximates the solutions of general fourth-order PDEs. By exploiting the underlying mathematical structure of the system, we rigorously proved that the scheme satisfies a discrete energy inequality and guarantees positive numerical solutions under a specific CFL condition. Our numerical experiments demonstrate the robustness of the method, particularly in challenging scenarios where the thin-film height approaches zero—a regime that typically leads to simulation breakdown or unphysical numerical artifacts in standard methods.

While the present work establishes the fundamental structure-preserving properties of the scheme, the issues of rigorous convergence, stability analysis, and error estimation remain open. We plan to address these aspects in an upcoming paper by leveraging the discrete energy stability established herein.

Currently, the proposed scheme is formally first-order accurate, which inherently introduces significant numerical diffusion, particularly on coarse grids. While this restricts its immediate applicability for high-fidelity simulations of complex thin-film flows, it serves as a crucial foundational step. Ultimately, this work establishes the theoretical framework necessary for developing higher-order, structure-preserving IMEX methods for these relaxation systems. The higher-order IMEX methods can be promising for reducing the computational costs significantly while capturing complex, physically meaningful dynamics in a reliable manner.

\bibliographystyle{mystyle}
\bibliography{references}

\end{document}